\documentclass[amssymb,amsfonts,refcheck,12pt,verbatim,righttag]{amsart}

\usepackage{graphicx}
\usepackage{graphics,color}
\usepackage{amssymb,mathrsfs,hyperref}
\usepackage{graphicx,color,tikz,caption,subcaption}
\usepackage[all]{xy}

\numberwithin{equation}{section} 
\theoremstyle{plain}

\newtheorem{thm}{Theorem}[section]
\newtheorem{lem}[thm]{Lemma}
\newtheorem{pro}[thm]{Proposition}

\newtheorem{de}[thm]{Definition}
\newtheorem{rem}[thm]{Remark}

\def\R {{\Bbb R}}
\def\N {{\Bbb N}}

\def\M {{\mathcal M}}

\def\A {{\mathcal A}}

\usepackage{float}
\begin{document}
\baselineskip 14pt
\title{Dimension estimates for $C^1$ iterated function systems and  repellers. Part I}

\author{De-Jun Feng}
\address{Department of Mathematics\\ The Chinese University of Hong Kong\\ Shatin,  Hong Kong\\ }
\email{djfeng@math.cuhk.edu.hk}

\author{K\'aroly Simon}
\address[K\'aroly Simon]{Budapest University of Technology and Economics, Department of Stochastics, Institute of Mathematics and
MTA-BME Stochastics Research Group,
 1521 Budapest, P.O.Box 91, Hungary} \email{simonk@math.bme.hu}


\date{}

\begin{abstract}
This is the first article in a two-part series containing some results on dimension estimates for $C^1$ iterated function systems and  repellers. In this part, we prove that the
upper box-counting dimension of the attractor of any $C^1$ iterated function system (IFS) on ${\Bbb R}^d$ is bounded above by its singularity dimension, and the upper packing dimension of any ergodic invariant measure associated with this IFS is bounded above by its Lyapunov dimension. Similar results are obtained for the repellers for $C^1$ expanding maps on Riemannian manifolds.
\end{abstract}

\maketitle
\section{Introduction}\label{S-1}

The dimension theory of iterated function systems (IFS) and dynamical repellers has developed into an important field of research during the last 40 years.  One of the main objectives is to estimate variant notions of dimension of the involved invariant sets and measures.
Despite many new and significant developments in recent years, only the cases of conformal repellers and attractors of conformal IFSs under certain separation condition have been completely understood. In such cases, the topological pressure plays a crucial role in the theory. Indeed, the Hausdorff and box-counting dimensions of the repeller $X$ for a $C^1$ conformal expanding map $f$ are given by the unique number $s$ satisfying the Bowen-Ruelle formula $P(X, f, -s\log \|D_xf\|)=0$, where the functional $P$ is the topological pressure, see \cite{Bowen1979,Ruelle1982,GatzourasPeres1997}. A similar formula is obtained for the Hausdorff and box-counting dimensions of the attractor of a conformal IFS satisfying the open set condition (see e.g. \cite{Patzschke1997}).

The study of dimension in the non-conformal cases has proved to be much more difficult.  In his seminal paper \cite{Falconer88},  Falconer established a general upper bound on the Hausdorff  and box-counting dimensions of a self-affine set (which is the attractor of an IFS consisting of contracting affine maps) in terms of the so-called affinity dimension, and proved that for typical self-affine sets under a mild assumption this upper bound is equal to the dimension. So far substantial progresses have been made towards understanding when the  Hausdorff  and box-counting dimensions of a concrete planar self-affine set are equal to its affinity dimension; see \cite{BaranyHochmanRapaport2019,HochmanRapaport2019} and the references therein. However very little has been known in the higher dimensional case.

In \cite[Theorem 5.3]{Falconer94}, by developing a sub-additive version of the thermodynamical formalism, Falconer showed that the upper box-counting dimension of a mixing repeller $\Lambda$
of a non-conformal $C^2$ mapping $\psi$, under the following  distortion condition
\begin{equation}
\label{e-bunch}
\|(D_x\psi)^{-1}\|^2 \|D_x\psi\|<1 \quad \mbox{ for }x\in \Lambda,
\end{equation}
is bounded
above by the zero point of the (sub-additive) topological pressure associated with the singular value functions of the derivatives of iterates of $\psi$.  We write this zero point as $\dim_{S^*}\Lambda$ and call it the {\em singularity dimension of $\Lambda$} (see Definition \ref{de-6.1} for the detail). The condition \eqref{e-bunch} is used to prove the bounded distortion property of the singular value functions, which enables one to control the distortion of balls after many iterations (see \cite[Lemma 5.2]{Falconer88}).
Examples involved ``triangular maps'' were constructed in \cite{ManningSimon07} to show that  the condition \eqref{e-bunch} is necessary for this bounded distortion property.

Using a quite different approach,  Zhang \cite{Zhang97} proved that the Hausdorff dimension of the repeller of an arbitrary $C^1$ expanding map $\psi$ is also bounded above by the singularity dimension.  (We remark that the upper bound given by Zhang was defined in a slightly different way, but it is equal to the singularity dimension; see \cite[Corollary 2]{BanCaoHu10} for a proof.) The basic technique used in Zhang's proof is to estimate the Hausdorff measure of $\psi(A)$ for small sets $A$ (see \cite[Lemma 3]{Zhang97}), which is applied to only one iteration so it  avoids assuming any distortion condition.  However his method does not apply  to  the box-counting dimension.

Thanks to the results of Falconer and Zhang, it arises a natural question whether the upper box-counting dimension of a $C^1$  repeller is always bounded above by its singularity dimension.  The challenge here is the lack of valid tools to analyse the geometry of  the images of balls under  a large number of iterations of $C^1$ maps.   In \cite[Theorem 3]{Barreira03}, Barreira made a positive claim to this question, but  his proof  contains a crucial mistake \footnote{This result was cited/applied in several papers (e.g., \cite[Corollary 4]{BanCaoHu10}, \cite[Theorems 4.4-4.7]{ChenPesin2010}) without noticing the  mistake in \cite{Barreira03}.},   as found by Manning and Simon \cite{ManningSimon07}. In a recent paper \cite[Theorem 3.2]{CaoPesinZhao2019}, Cao, Pesin and Zhao obtained an upper bound for the upper box-counting dimension of the repellers of  $C^{1+\alpha}$ maps satisfying certain dominated splitting property. However that upper bound depends on the involved splitting and  is usually strictly larger than the singularity dimension.

In the present paper, we give an affirmative answer to the above question. We also establish an analogous  result for the attractors of $C^1$ non-conformal IFSs. Meanwhile we prove that  the upper packing dimension of an ergodic invariant measure supported on a $C^1$ repeller (resp. the projection of an ergodic measure on the attractor of a $C^1$  IFS) is bound above by its Lyapunov dimension.

In the continuation of this paper \cite{FengSimon2020} we verify that upper bound estimates of the dimensions of the attractors and ergodic measures in the previous paragraph, give the exact values of the dimensions  for some families of $C^2$ non-conformal IFSs on the plane at least typically. Typically means that the assertions hold for almost all translations of the system. These families include the
 $C^2$ non-conformal IFSs on the plane for which all differentials are either diagonal matrices, or  all  differentials are lower triangular matrices and the contraction in $y$ direction is stronger than in $x$ direction.

We first state our results for $C^1$ IFSs. To this end, let us introduce some notation and definitions. Let $Z$ be a compact subset of $\R^d$. A finite family $\{f_i\}_{i=1}^\ell$ of contracting self maps on $Z$ is called a {\it $C^1$ iterated function system} (IFS), if there exists an open set $U\supset Z$ such that  each $f_i$ extends to a contracting $C^1$-diffeomorpism $f_i:\; U\to f_i(U)\subset U$.
Let $K$ be the attractor of the IFS, that is,  $K$ is the unique non-empty compact subset of ${\Bbb R}^d$ such that
\begin{equation}
\label{e-IFS}
K=\bigcup_{i=1}^\ell f_i(K)
\end{equation}
(cf.~\cite{Falconer2003}).

Let $(\Sigma,\sigma)$ be the one-sided full shift over the alphabet $\{1,\ldots, \ell\}$ (cf.~\cite{Bowen08}).    Let $\Pi: \Sigma\to K$ denote the canonical coding map associated with the IFS $\{f_i\}_{i=1}^\ell$. That is,
\begin{equation}
\label{e-coding}
\Pi(x)=\lim_{n\to \infty} f_{x_1}\circ \cdots\circ f_{x_n}(z), \qquad x=(x_n)_{n=1}^\infty,
\end{equation}
with $z\in U$. The definition of $\Pi$ is independent of the choice of $z$.

For any  compact subset $X$ of $\Sigma$ with $\sigma X\subset X$, we call $(X,\sigma)$ a {\it one-sided subshift} or simply {\it subshift} over $\{1,\ldots, \ell\}$ and let $\dim_SX$ denote the singularity dimension of $X$ with respect to the IFS $\{f_i\}_{i=1}^\ell$ (cf. Definition \ref{de-1}).

For a set $E\subset \R^d$, let $\overline{\dim}_BE$ denote the upper box-counting dimension of $E$ (cf.~\cite{Falconer2003}).  The first result in this paper is the following,  stating that  the upper box-counting dimension of $\Pi(X)$ is bounded above by the singularity dimension of $X$.

\begin{thm}
\label{thm-1}
Let $X\subset \Sigma$ be compact and $\sigma X\subset X$. Then $\overline{\dim}_B\Pi(X)\leq \dim_SX$. In particular,
$$\overline{\dim}_BK\leq \dim_S\Sigma.$$
\end{thm}
For an ergodic $\sigma$-invariant measure $m$ on $\Sigma$, let $\dim_{L}m$ denote the Lyapunov dimension of $m$ with respect to $\{f_i\}_{i=1}^\ell$ (cf. Definition \ref{de-2}).
 For a Borel probability measure $\eta$ on $\R^d$ (or a manifold), let $\overline{\dim}_P\eta$ denote the upper packing dimension of $\eta$. That is,
 $$
 \overline{\dim}_P\eta={\rm esssup}_{x\in {\rm supp}(\eta)}\overline{d}(\eta,x), \qquad \mbox{with }\overline{d}(\eta, x):=\limsup_{r\to 0}\frac{\log \eta(B(x,r))}{\log r},
 $$
 where $B(x,r)$ denotes the closed ball  centered at $x$ of radius $r$. Equivalently,
 $$\overline{\dim}_P\eta=\inf\{ \dim_PF:\; \mbox{ $F$ is a Borel set with } \eta(F)=1\},
 $$
 where $\dim_PF$ stands for  the packing dimension of $F$ (cf.~\cite{Falconer2003}).
 See  e.g.~\cite{FanLauRao02} for a proof.

 Our second result is the following, which can be viewed as a measure analogue of Theorem \ref{thm-1}.

\begin{thm}
\label{thm-2}
 Let $m$ be an ergodic $\sigma$-invariant measure on $\Sigma$, then
$$\overline{\dim}_P(m\circ \Pi^{-1})\leq \dim_{L}m,$$
where $m\circ \Pi^{-1}$ stands for the push-forward of $m$ by  $\Pi$.
\end{thm}

The above theorem improves a result of Jordan and Pollicott \cite[Theorem 1]{JordanPollicott08} which  states that $$\overline{\dim}_H(m\circ \Pi^{-1})\leq \dim_{L}m$$ under a slightly more general setting, where $\overline{\dim}_H(m\circ \Pi^{-1})$ stands for the upper Hausdorff dimension of $m\circ \Pi^{-1}$. Recall that the upper Hausdorff dimension of a measure is the infimum of the Hausdorff dimension of Borel sets of full measure, which is always less than or equal to the upper packing dimension of the measure.

Next we turn to the case of repellers.    Let ${\pmb M}$ be a smooth Riemannian manifold of dimension $d$ and $\psi: {\pmb M}\to {\pmb M}$ a $C^1$-map. Let $\Lambda$ be a compact subset of ${\pmb M}$ such that $\psi(\Lambda)=\Lambda$. We say that $\psi$ is {\it expanding} on $\Lambda$ and $\Lambda$ a {\it repeller} if
 \begin{itemize}
 \item[(a)] there exists $\lambda>1$ such that $\|(D_z\psi) v\|\geq \lambda\|v\|$ for all $z\in \Lambda$, $v\in T_z{\pmb M}$ (with respect to a Riemannian metric on ${\pmb M}$);
 \item[(b)] there exists an open neighborhood $V$ of $\Lambda$ such that
 $$
 \Lambda=\{z\in V:\; \psi^n(z)\in V \mbox{ for all } n\geq 0\}.
 $$
 \end{itemize}
We refer the reader to \cite[Section 20]{Pesin1997} for more details.
In what follows we always assume that $\Lambda$ is a repeller of $\psi$. Let $\dim_{S^*}\Lambda$ denote the singlular dimension of $\Lambda$ with respect to $\psi$ (see Definition \ref{de-6.1}). For an ergodic $\psi$-invariant measure $\mu$ on $\Lambda$, let $\dim_{L^*}\mu$ be the Lyapunov dimension of $\mu$ with respect to $\psi$ (see Definition \ref{de-6.2}).
Analogous to Theorems \ref{thm-1}-\ref{thm-2}, we have the following.

\begin{thm}
\label{thm-3}  Let $\Lambda$ be the repeller of  $\psi$. Then
$$\overline{\dim}_B \Lambda\leq \dim_{S^*}\Lambda.$$
\end{thm}

\begin{thm}
\label{thm-4} Let $\mu$ be an ergodic $\psi$-invariant measure supported on $\Lambda$. Then
$$
\overline{\dim}_P\mu\leq \dim_{L^*}\mu.
$$
\end{thm}

For the estimates of the box-counting dimension of attractors of $C^1$ non-conformal IFSs (resp. $C^1$ repellers), the reader may reasonably ask what difficulties arose in the previous work \cite{Falconer94} which the present article overcomes.  Below we give an explanation and illustrate roughly our strategy of the proof.

Let us give an account of the IFS case. The case of repellers is similar.  Let $K$ be the attractor of a $C^1$ IFS $\{f_i\}_{i=1}^\ell$. To estimate $\overline{\dim}_BK$, by definition one needs to estimate for given $r>0$ the least number of balls of radius $r$ required to cover $K$, say $N_r(K)$. To this end, one may iterate the IFS  to get $K=\bigcup_{i_1\ldots i_n} f_{i_1\cdots i_n}(K)$ and then estimate $N_r(f_{i_1\cdots i_n}(K))$ separately, where $f_{i_1\ldots i_n}:=f_{i_1}\circ\cdots\circ f_{i_n}$. For this purpose, one needs to estimate $N_r(f_{i_1\cdots i_n}(B))$, where  $B$ is a fixed ball covering $K$.

Under the strong assumption of  distortion property, Falconer was able to show that $f_{i_1\cdots i_n}(B)$ is roughly comparable to the  ellipsoid $(D_x f_{i_1\cdots i_n})(B)$ for each $x\in B$ (see \cite[Lemma 5.2]{Falconer94}), then by cutting the ellipsoid into roughly round pieces he could use certain singular value function to give an upper bound of $N_r(f_{i_1\cdots i_n}(B))$, and then apply the sub-additive thermodynamic formalism to estimate the growth rate of $\sum_{i_1\cdots i_n}N_r(f_{i_1\cdots i_n}(B))$.  However in the general $C^1$ non-conformal case, this approach is no longer feasible, since it seems hopeless to analyse the geometric shape of $f_{i_1\cdots i_n}(B)$ when $n$ is large.

The strategy of our approach is quite different. We use an observation going back to  Douady and  Oesterl\'{e} \cite{DouadyOesterle80} (see also \cite{Zhang97}) that,  for a given $C^1$ map $f$, when $B_0$ is an enough small ball in a fixed bounded region, then $f(B_0)$ is close to be an ellipsoid and so it can be covered by certain number of balls controlled by the singular values of the differentials of $f$ (see Lemma \ref{lem-1}). Since the maps $f_i$ in the IFS are contracting, we may apply this fact to the maps $f_{i_n}$, $f_{i_{n-1}}$, \ldots, $f_{i_1}$ recursively.  Roughly speaking, suppose that $B_0$ is a ball of small radius $r_0$, then  $f_{i_n}(B_0)$ can be covered by $N_1$ balls of radius $r_1$,  and the image of each of them under $f_{i_{n-1}}$ can be covered by $N_2$ balls of radius of $r_2$, and so on, where $N_j, r_j/r_{j-1}$ ($j=1,\ldots, n$) can be controlled by the singular values of the differetials of $f_{i_{n-j+1}}$. In this way we get an estimate that $N_{r_n}(f_{i_1\cdots i_n}(B_0))\leq N_1\ldots N_n$ (see Proposition \ref{pro-2} for a more precise statement), which is in spirit analogous to the corresponding estimate for the Hausdorff measure by Zhang \cite{Zhang97}.  In this process,  we don't need to consider the differentials of $f_{i_1\cdots i_n}$ and so no distortion property is required.  By developing a key technique  from the thermodynamic formalism (see Proposition \ref{pro-2.4}), we can get an upper bound for $\overline{\dim}_BK$, say $s_1$. Replacing the IFS $\{f_i\}_{i=1}^\ell$ by its $n$-th iteration $\{f_{i_1\ldots i_n}\}$, we get other upper bounds $s_n$. Again using some technique in the thermodynamic formalism (see Proposition \ref{pro-1}), we  manage to show that $\lim\inf s_n$ is bounded above by the singularity dimension.

The paper is organized as follows. In Section \ref{S-2}, we give some preliminaries about the sub-additive thermodynamic formalism and give the definitions of the singularity and Lyapunov dimensions with respect to a $C^1$ IFS. In  Section \ref{S-3}, we prove two auxiliary results (Propositions \ref{pro-1} and \ref{pro-2.4}) which play a key role in the proof of Theorem \ref{thm-1} (and that of Theorem \ref{thm-3}). The proofs of Theorems \ref{thm-1} and \ref{thm-2} are given in Sections \ref{S-4}-\ref{S-5}, respectively. In Section \ref{S-6}, we give the definitions of the singularity and Lyapunov dimensions in the repeller case and prove Theorems \ref{thm-3}-\ref{thm-4}. For the convenience of the reader,  in the appendix we summarize the main notation and typographical conventions used in this paper.

\section{Preliminaries}
\label{S-2}
\subsection{Variational principle for sub-additive pressure}
\label{subsec:TF}
In order to define the singularity and Lyapunov dimensions and prove our main results,  we require some elements from the sub-additive thermodynamic formalism.

 Let  $(X,d)$ be a compact metric space and
   $T:X\to X$ a continuous mapping. We call $(X,T)$ a {\it topological dynamical system}. For $x,y\in X$ and $n\in \N$, we define
   \begin{equation}\label{091}
     d_n(x,y):=\max\limits_{0 \leq i \leq n-1}d(T^i(x),T^i(y)).
   \end{equation}
  A set $E \subset X$ is called {\it $(n,\varepsilon)$-separated} if for every distinct $x,y\in E$ we have $d_n(x,y)>\varepsilon$.

Let $C(X)$ denote the set of real-valued continuous functions on $X$. Let $\mathcal{G}=\left\{g_n\right\}_{n=1}^{\infty }$ be a {\it sub-additive potential} on $X$, that is,  $g_n\in C(X)$ for all $n\geq 1$ such that
  \begin{equation}\label{090}
    g_{m+n}(x) \leq g_{n}(x)+g_{m}(T^nx) \quad\mbox{ for  all }  x\in X \mbox { and }n,m\in \N.
  \end{equation}
 Following  \cite{CFH08}, below we define the topological pressure of $\mathcal{G}$.
\begin{de}\label{e-pressure}
{\rm
 For given $n\in\mathbb{N}$ and $\varepsilon>0$ we define
\begin{equation}\label{089}
  P_n(X, T, \mathcal{G},\varepsilon):=
  \sup\left\{
  \sum\limits_{x\in E}
  \exp (g_n(x)):\;
  E \mbox{ is  an $(n,\varepsilon)$-separated set }
  \right\}.
\end{equation}
Then the  {\it topological pressure of $\mathcal{G}$ with respect to $T$}
is defined by
\begin{equation}\label{088}
  P(X, T, \mathcal{G}):=
  \lim\limits_{\varepsilon\to0}
  \limsup\limits_{n\to\infty}
  \frac{1}{n}\log  P_n(X,T,\mathcal{G},\varepsilon).
\end{equation}
}
\end{de}

If the potential $\mathcal G$ is additive, i.e. $g_n=S_ng:=\sum_{k=0}^{n-1}g\circ T^k$ for some  $g\in C(X)$, then $P(X, T, \mathcal G)$ recovers the classical topological pressure  $P(X, T, g)$ of $g$ (see e.g.~\cite{Walters82}).

Let $\mathcal{M}(X)$ denote the set of Borel probability measures on $X$, and  $\mathcal{M}(X,T)$ the set of  $T$-invariant Borel probability measures on $X$.
For $\mu\in \mathcal{M}(X,T)$, let $h_\mu(T)$  denote the measure-theoretic entropy of $\mu$ with respect to $T$ (cf. \cite{Walters82}). Moreover, for  $\mu \in\mathcal M(X, T)$, by sub-additivity we have
 \begin{equation}
 \label{e-N1}
 \mathcal{G}_*(\mu):=\lim_{n\to\infty} \frac{1}{n} \int g_n d\mu=\inf_n  \frac{1}{n} \int g_n d\mu\in [-\infty,\infty).
 \end{equation}
 See e.g.~\cite[Theorem 10.1]{Walters82}.  We call  $\mathcal{G}_*(\mu)$  the {\it Lyapunov exponent} of $\mathcal{G}$ with respect to $\mu$.

The proofs of our main results rely on the following general variational principle for the topological pressure of sub-additive potentials.
\begin{thm} [\cite{CFH08}, Theorem 1.1]
\label{thm:sub-additive-VP}  
Let $\mathcal{G}=\left\{g_n\right\}_{n=1}^{\infty }$ be a sub-additive potential on a topological dynamical system $(X,T)$. Suppose that the topological entropy of $T$ is finite. Then
  \begin{equation}\label{variational-principle}
    P(X, T,\mathcal{G})=
    \sup\left\{h_\mu(T)+\mathcal{G}_*(\mu):\;
    \mu\in\mathcal{M}(X,T)
    \right\}.
  \end{equation}
\end{thm}
Particular cases of the above result, under stronger assumptions on the dynamical systems and the potentials, were previously obtained by many authors, see for example \cite{Falconer88b, FengLau2002, Feng2004, Kaenmaki04, Mummert06, Barreira10} and references therein.

Measures that achieve the supremum in \eqref{variational-principle}  are called  {\em  equilibrium measures} for the potential ${\mathcal G}$. There exists at least one ergodic equilibrium measure when the entropy map $\mu\mapsto h_\mu(T)$ is upper semi-continuous; this is the case when $(X,T)$ is a subshift,    see e.g.~\cite[Proposition 3.5]{Feng11} and the remark there.

The following well-known result is also needed in our proofs.
\begin{lem}
\label{lem-2.3}  Let $X_i, i= 1,2$ be compact metric spaces and let $T_i:
X_i \to X_i$ be continuous. Suppose $\pi: X_1 \to X_2$ is a
continuous surjection such that the following diagram commutes:
\begin{equation*}
  \xymatrix{ X_1 \ar[r]^{T_1}  \ar[d]_\pi   & %
                   X_1\ar[d]^{\pi}\\ %
X_2\ar[r]_{T_2} & X_2  }%
\end{equation*}
 Then $\pi_*:\;{\mathcal M}(X_1, T_1)\rightarrow {\mathcal M}(X_2, T_2)$
(defined by $\mu\mapsto \mu\circ \pi^{-1}$) is surjective.
If furthermore there is an integer $q>0$ so that $\pi^{-1}(y)$ has at
most $q$ elements for each $y\in X_2$, then
$$
h_{\mu}(T_1)=h_{\mu\circ \pi^{-1}}(T_2)
$$
for each $\mu\in {\mathcal M}(X_1, T_1)$.
\end{lem}

\begin{proof} The first part of the result is the same as
\cite[Chapter IV, Lemma 8.3]{Mane1987}. The second part follows from
the Abramov-Rokhlin formula (see \cite{Bogenschutz1992}).
\end{proof}

\subsection{Subshifts}
In this subsection, we introduce some basic notation and definitions
about subshifts.

Let $(\Sigma,\sigma)$ be the one-sided full shift over the alphabet $\mathcal A=\{1,\ldots, \ell\}$. That is, $\Sigma=\mathcal A^{\Bbb N}$ endowed with the product topology, and $\sigma:\Sigma\to \Sigma$ is the left shift defined by $(x_i)_{i=1}^\infty\mapsto (x_{i+1})_{i=1}^\infty$.   The topology of $\Sigma$ is compatible with the following metric on $\Sigma$:
$$
d(x,y)=2^{-\inf\{k:\; x_{k+1}\neq y_{k+1}\}} \quad \mbox{ for }x=(x_i)_{i=1}^\infty,\; y=(y_i)_{i=1}^\infty.
$$
 For $x=(x_i)_{i=1}^\infty\in \Sigma$ and $n\in {\Bbb N}$, write $x|n=x_1\ldots x_n$.

Let $X$ be a non-empty compact subset of $\Sigma$  satisfying  $\sigma X\subset X$. We call $(X, \sigma)$  a {\it one-sided subshift} or simply {\it subshift} over  $\mathcal A$. We denote the collection  of finite words allowed in $X$ by $X^*$, and the subset of $X^*$ of  words of length $n$ by $X_n^*$. In particular, define for $n\in {\Bbb N}$,
\begin{equation}
\label{e-xn}
X^{(n)}:=\left\{(x_i)_{i=1}^\infty\in \mathcal A^{\Bbb N}:\; x_{kn+1}x_{kn+2}\ldots x_{(k+1)n}\in X_n^* \; \mbox { for all } k\geq 0\right\}.
\end{equation}

 Let  $\mathcal G=\{ g_n\}_{n=1}^\infty$  be a  sub-additive potential on a subshift $(X,\sigma)$. It is known that in such case, the topological pressure of $\mathcal G$ can be alternatively defined by
 \begin{equation}
 \label{e-Falconer}
 P(X,\sigma, \mathcal G)= \lim_{n\to\infty} \frac{1}{n}\log\left( \sum_{\mathbf{i}\in X_n^*} \sup_{x\in [\mathbf{i}]\cap X} \exp(g_n(x)) \right),
\end{equation}
where $[\mathbf{i}]:=\{x\in \Sigma:\; x|n=\mathbf{i}\}$ for $\mathbf{i}\in \mathcal A^n$; see \cite[p.~649]{CFH08}.
The limit can be seen to exist by using a standard sub-additivity argument. We remark that \eqref{e-Falconer} was first introduced by Falconer in \cite{Falconer88b} for the definition of the topological pressure of sub-additive potentials on a mixing repeller.

\subsection{Singularity dimension and Lyapunov dimension with respect to  $C^1$ IFSs}
\label{S-2.3}
In this subsection, we define the singularity  and Lyapunov dimensions with respect to $C^1$ IFSs. The corresponding definitions with respect to $C^1$ repellers will be given in Section~\ref{S-5}.

Let $\{f_i\}_{i=1}^\ell$ be a $C^1$ IFS on $\R^d$ with attractor $K$.  Let $(\Sigma,\sigma)$ be the one-sided full shift over the alphabet $\{1,\ldots, \ell\}$ and let $\Pi: \Sigma\to K$ denote the corresponding coding map defined as in \eqref{e-coding}. For a differentiable function $f:\; U\subset \R^d\to \R^d$, let $D_zf$ denote the differential of $f$ at $z\in U$.

For $T\in \R^{d\times d}$, let $\alpha_1(T)\geq\cdots\geq \alpha_d(T)$ denote the  singular values of $T$. Following \cite{Falconer88},  for $s\geq 0$ we define the {\it singular value function} $\phi^s:\; \R^{d \times d}\to [0,\infty)$ as
\begin{equation}
\label{e-singular}
\phi^s(T)=\left\{
\begin{array}
{ll}
\alpha_1(T)\cdots \alpha_k(T) \alpha_{k+1}^{s-k}(T) & \mbox{ if }0\leq s\leq d,\\
\det(T)^{s/d} & \mbox{ if } s>d,
\end{array}
\right.
\end{equation}
where $k=[s]$ is the integral part of $s$.

\begin{de}
\label{de-1}
{\rm
 For a compact   subset $X$ of $\Sigma$ with $\sigma(X)\subset X$, the {\it singularity dimension of $X$ with respect to $\{f_i\}_{i=1}^\ell$}, written as  $\dim_SX$, is the unique non-negative value $s$ for which
$$
P(X,\sigma, \mathcal G^s)=0,
$$
where $\mathcal G^s=\{g_n^s\}_{n=1}^\infty$ is the sub-additive potential on $\Sigma$ defined by
\begin{equation}
\label{e-gn}
g_n^s(x)=\log \phi^s(D_{\Pi\sigma^n x}f_{x|n}), \quad x\in \Sigma,
\end{equation}
with $f_{x|n}:=f_{x_1}\circ \cdots\circ f_{x_n}$ for $x=(x_n)_{n=1}^\infty$.
}
\end{de}

\begin{de}
\label{de-2}
{\rm
Let $m$ be an ergodic $\sigma$-invariant Borel probability measure on $\Sigma$. For any $i\in\{1, \ldots ,d\}$, the $i$-th Lyapunov exponent of $m$ is
\begin{equation}\label{e-lambda_1}
\lambda_i(m):=\lim_{n\to \infty} \frac{1}{n}\int \log\left( \alpha_i(D_{\Pi\sigma^nx}f_{x|n})\right)\; dm(x).
\end{equation}
The {\it Lyapunov dimension of $m$ with respect to $\{f_i\}_{i=1}^\ell$}, written as $\dim_{\rm L}m$, is the unique non-negative value $s$ for which
$$
h_m(\sigma)+ \mathcal G^s_*(m)=0,
$$
where $\mathcal G^s=\{g^s_n\}_{n=1}^\infty$ is defined as in  \eqref{e-gn} and $\mathcal G^s_*(m):=\lim_{n\to \infty} \frac{1}{n}\int g^s_n\; dm$. See Figure \ref{figure-1} for the mapping $s\mapsto -\mathcal G^s_*(m)$ in the case when $d=2$.
}
\end{de}

It follows from the definition of the singular value function $\phi^s$
that for an ergodic measure $m$ we have
\begin{equation*}\label{e-eq6}
  \mathcal G^s_*(m)=
  \left\{
    \begin{array}{ll}
    \lambda_1(m)+\cdots+\lambda_{[s]}(m)+(s-[s])\lambda_{[s]+1}(m)  , &
\hbox{if $s < d$} \\
 \frac{s}{d} (\lambda_1(m)+\cdots+\lambda_{d}(m) )   , & \hbox{if $s \geq d$.}
    \end{array}
  \right.
\end{equation*}
 Observe that in the special case when all the Lyapunov exponents are equal to the same $\lambda$ then $\dim_{\rm L}m=\frac{h_m(\sigma)}{-\lambda}$.

\begin{figure}
  \centering
  \includegraphics[width=10cm]{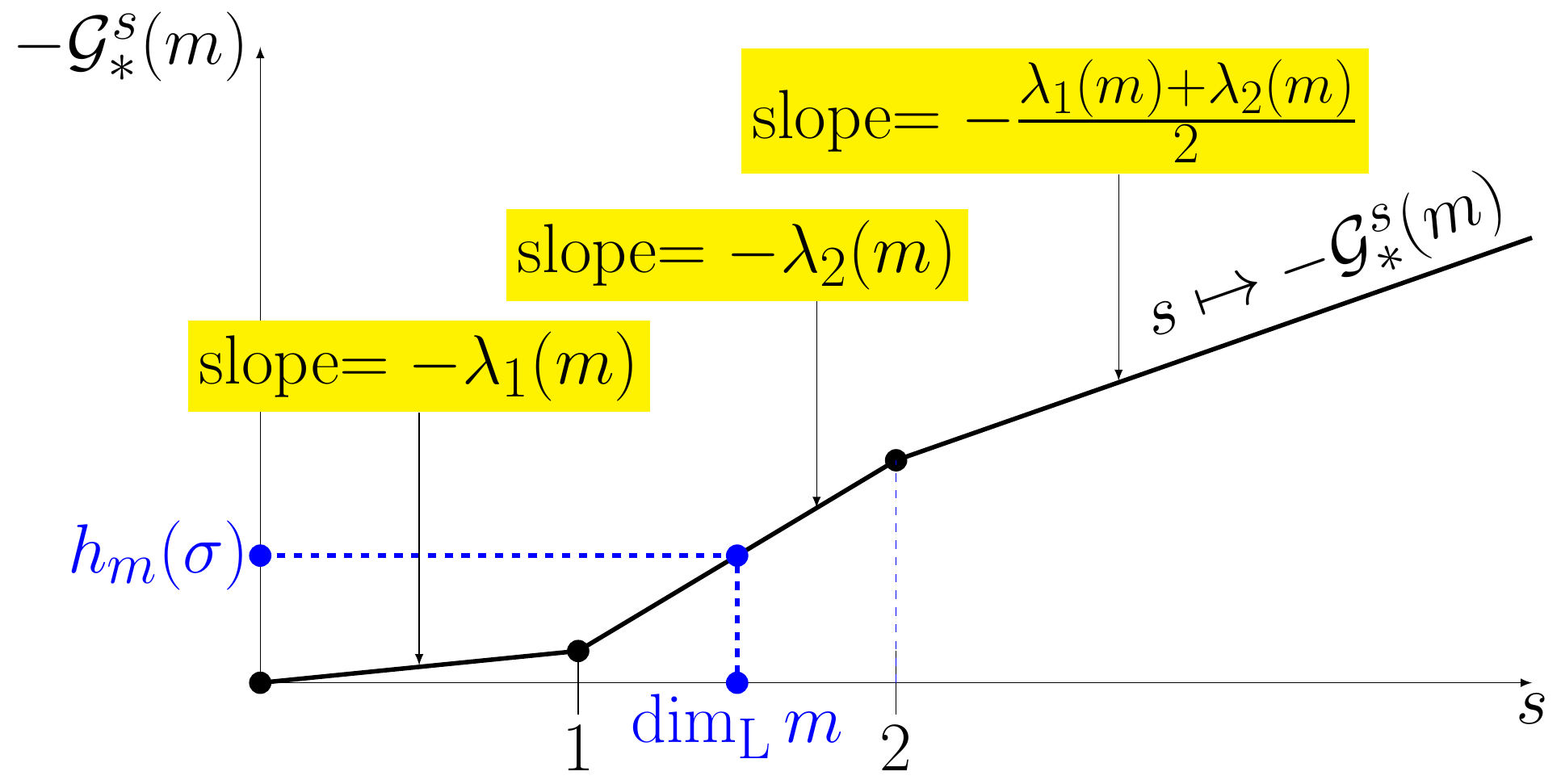}
  \caption{The connection between Lyapunov dimension, entropy and the function  $s\mapsto -\mathcal{G}^s_*(m)$  when $d=2$.}\label{figure-1}
\end{figure}

\begin{rem}
{\rm
\begin{itemize}
\item[(i)] The concept of singularity dimension was first introduced by Falconer \cite{Falconer88, Falconer94}; see also \cite{KaenmakiVilppolainen10}. It is also called {\it affinity dimension} when the IFS $\{f_i\}_{i=1}^\ell$ is affine, i.e.~each map $f_i$ is affine.
\item[(ii)] The definition of Lyapunov dimension of ergodic measures with respect to an IFS  presented above was taken from \cite{JordanPollicott08}. It is a  generalization of that given in \cite{JordanPollicottSimon07} for affine IFSs.
\end{itemize}
}
\end{rem}

\subsection{A special consequence of Kingman's subadditive ergodic theorem}
Here we state a special consequence of Kingman's subadditive ergodic theorem which will be  needed in the proof of Lemma  \ref{lem-4.1}.
\begin{lem}
\label{lem-2.3.1}
Let $T$ be a measure preserving transformation of the probability space $(X, {\mathcal B}, m)$, and
 let $\{g_{n}\}_{{n\in {\mathbb  {N}}}}$ be a sequence of $L^{1}$ functions satisfying the following sub-additivity relation:
 $$g_{{n+m}}(x)\leq g_{n}(x)+g_{m}(T^{n}x) \quad \mbox{ for all }x\in X.$$
  Suppose that there exists $C>0$ such that
\begin{equation}
\label{e-2.3.1}
|g_n(x)|\leq Cn \quad \mbox{ for all $x\in X$ and $n\in {\Bbb N}$}.
\end{equation}
 Then
$$
\lim _{n\to \infty }E\left(\frac {g_{n}}{n}|{\mathcal C}_n\right)(x)=g(x):=\lim_{n\to \infty}\frac{g_n(x)}{n}
$$
for $m$-a.e.~$x$, where
$${\mathcal C}_n:=\{B\in {\mathcal B}:\; T^{-n}B =B \; a.e.\},$$
 $E(\cdot|\cdot)$ denotes the conditional expectation and $g(x)$ is $T$-invariant.
\end{lem}
\begin{proof}
By Kingman's subadditive ergodic theorem, $\displaystyle\frac{g_n}{n}$ converges pointwisely to a $T$-invariant function $g$ a.e.  In the meantime, by Birkhoff's ergodic theorem,
for each $n\in {\Bbb N}$,
$$
E(g_n|{\mathcal C}_n)(x)=\lim_{k\to \infty}\frac{1}{k}\sum_{j=0}^{k-1}g_n(T^{jn}x) \mbox{ a.e.}
$$
Since   $\sum_{j=0}^{k-1}g_n(T^{jn}x)\geq g_{kn}(x)$ by subadditivity,   it follows that
\begin{equation}
\label{e-2.3.2}
E\left(\frac {g_{n}}{n}\Big|{\mathcal C}_n\right)(x)\geq \lim_{k\to \infty} \frac{g_{nk}(x)}{nk} = g(x) \mbox{ a.e.}
\end{equation}
Since the sequence $\{g_n\}$ is sub-additive and satisfies \eqref{e-2.3.1}, by \cite[Lemma 2.2]{Kaenmaki04}, for any $0< k< n$,
$$
\frac{g_n(x)}{n}\leq \frac{1}{kn}\sum_{j=0}^{n-1} g_k(T^jx)+\frac{3kC}{n},\quad x\in X.
$$
As a consequence,
\begin{equation}
\label{e-2.3.3}
E\left(\frac {g_{n}}{n}\Big|{\mathcal C}_n\right)(x)\leq \frac{1}{kn}E\left(\sum_{j=0}^{n-1}g_{k}\circ T^j\Big|{\mathcal C}_n\right)(x)+\frac{3kC}{n}.
\end{equation}
Notice that for each $f\in L^1$ and $n\in {\Bbb N}$,
\begin{equation}
\label{e-2.3.4}
E\left(\sum_{j=0}^{n-1}f\circ T^j\Big|{\mathcal C}_n\right)=nE(f|{\mathcal C}_1) \mbox{ a.e.}
\end{equation}
To see the above identity, one simply applies Birkhoff's ergodic theorem (with respect to the transformations $T^n$ and $T$, respectively) to the following limits
$$
\lim_{k\to \infty}\frac{1}{k}\sum_{j=0}^{k-1}\left(f+f\circ T+\cdots+f\circ T^{n-1}\right)(T^{nj}x)
=\lim_{k\to \infty} \frac{1}{k}\sum_{j=0}^{nk-1}f(T^jx).$$
Now applying the identity \eqref{e-2.3.4} (in which taking $f=g_k$) to  \eqref{e-2.3.3} yields
$$
E\left(\frac {g_{n}}{n}\Big|{\mathcal C}_n\right)\leq E\left(\frac {g_{k}}{k}\Big|{\mathcal C}_1\right) +\frac{3kC}{n} \mbox{ a.e.} \mbox{ for }0<k<n.
$$
It follows that $$\limsup_{n\to  \infty}E\left(\frac {g_{n}}{n}\Big|{\mathcal C}_n\right) \leq E\left(\frac {g_{k}}{k}\Big|{\mathcal C}_1\right) \mbox { a.e.~ for each $k$},
$$
 so by the dominated convergence theorem,
$$
\limsup_{n\to  \infty}E\left(\frac {g_{n}}{n}\Big|{\mathcal C}_n\right) \leq \lim_{k\to \infty} E\left(\frac {g_{k}}{k}\Big|{\mathcal C}_1\right)=
E(g|{\mathcal C}_1)=g \mbox{ a.e.}
$$
Combining it with \eqref{e-2.3.2} yields the desired result $\lim_{n\to  \infty}E\left(\frac {g_{n}}{n}\Big|{\mathcal C}_n\right)=g$ a.e.
\end{proof}

 \section{Some auxiliary results}
 \label{S-3}

In this section we give two auxiliary results (Propositions \ref{pro-1} and \ref{pro-2.4}) which are needed in the proof of Theorem \ref{thm-1}.
\begin{pro}
\label{pro-1}
Let $(X,\sigma)$ be a one-sided subshift over a finite alphabet ${\mathcal A}$ and  $\mathcal{G}=\{g_n\}_{n=1}^\infty$  a sub-additive potential on $\mathcal A^{\Bbb N}$.  Then
\begin{equation}
\label{e-t3}
P(X,\sigma, \mathcal G)=\lim_{n\to \infty}\frac{1}{n} P\left(X^{(n)},\sigma^n, g_n\right)=\inf_{n\geq 1}\frac{1}{n} P\left(X^{(n)},\sigma^n, g_n\right),
\end{equation}
where $X^{(n)}$ is defined as in \eqref{e-xn}.
\end{pro}

\begin{rem}
{\rm Instead of \eqref{e-t3}, it was proved in \cite[Proposition 2.2]{BanCaoHu10} that $$P(X,\sigma, \mathcal G)=\lim_{n\to \infty}\frac{1}{n} P\left(X,\sigma^n, g_n\right)$$ under a more general setting. We remark that the proof of \eqref{e-t3} is more subtle.}
\end{rem}

To prove the Proposistion \ref{pro-1}, we need the following.
\begin{lem}[\cite{CFH08}, Lemma 2.3]
\label{lem-3}  Under the assumptions of Proposition \ref{pro-1}, suppose that $\{\nu_n\}_{n=1}^\infty$ is a sequence in
$\M\left(\mathcal A^{\Bbb N}\right)$, where $\M({\mathcal A}^{\Bbb N})$ denotes the space of all Borel probability measures on $\mathcal A^{\Bbb N}$ with the weak* topology.
 We form the new sequence $\{\mu_n\}_{n=1}^\infty$ by
$\mu_n=\frac 1n \sum_{i=0}^{n-1}\nu_n\circ \sigma^{-i}$. Assume that
$\mu_{n_i}$ converges to $\mu$ in $\M({\mathcal A}^{\Bbb N})$ for some subsequence
$\{n_i\}$ of natural numbers. Then $\mu\in \M({\mathcal A}^{\Bbb N}, \sigma)$, and moreover
$$
\limsup_{i\to\infty}\frac{1}{n_i}\int
g_{n_i}(x)\;d\nu_{n_i}(x)\leq \mathcal G_*(\mu):=\lim_{n\to\infty}\frac{1}{n}\int g_n\;d\mu.
$$
\end{lem}

\begin{proof}[Proof of Proposition \ref{pro-1}]
 We first prove that for each $n\in {\Bbb N}$,
\begin{equation}
\label{e-t1}
P(X, \sigma, \mathcal G)\leq \frac{1}{n}P\left(X^{(n)},\sigma^n, g_n\right).
\end{equation}
To see this, fix $n\in {\Bbb N}$ and let $\mu$ be an equilibrium measure for the potential ${\mathcal G}$. Then
\begin{eqnarray*}
P(X, \sigma, \mathcal G)&=&h_\mu(\sigma)+\mathcal G_*(\mu)\\
&\leq & h_\mu(\sigma)+\frac{1}{n}\int g_n\; d\mu\qquad\mbox{ (by \eqref{e-N1})}\\
&=&\frac{1}{n}\left( h_\mu(\sigma^n)+\int g_n\; d\mu\right)\\
&\leq & \frac{1}{n}P\left(X^{(n)},\sigma^n, g_n\right),
\end{eqnarray*}
where in the last inequality, we use the fact that $\mu\in \mathcal M(X^{(n)}, \sigma^n)$ and the classical variational principle for the topological pressure of additive potentials. This proves \eqref{e-t1}.

In what follows we prove that
\begin{equation}
\label{e-t2}
P(X, \sigma, \mathcal G)\geq \limsup_{n\to \infty}\frac{1}{n}P\left(X^{(n)},\sigma^n, g_n\right).
\end{equation}
Clearly \eqref{e-t1} and \eqref{e-t2} imply \eqref{e-t3}. To prove \eqref{e-t2}, by the classical variational principle we can take a subsequence $\{n_i\}$ of natural numbers and $\nu_{n_i}\in \mathcal M\left(X^{(n_i)}, \sigma^{n_i}\right)$ such that
\begin{equation}\label{e-t4}
\limsup_{n\to \infty}\frac{1}{n}P\left(X^{(n)},\sigma^n, g_n\right)=\lim_{i\to \infty}\frac{1}{n_i} \left(h_{\nu_{n_i}}(\sigma^{n_i})+\int  g_{n_i}\; d\nu_{n_i}\right).
\end{equation}
Set $\mu^{(i)}=\frac{1}{n_i}\sum_{k=0}^{n_i-1}\nu_{n_i}\circ \sigma^{-k}$ for each $i$. Taking a subsequence if necessary we may assume that $\mu^{(i)}$ converges  to an element $\mu\in \mathcal M\left(\mathcal A^{\Bbb N}\right)$ in the weak* topology. By Lemma \ref{lem-3}, $\mu\in \M({\mathcal A}^{\Bbb N}, \sigma)$ and moreover
\begin{equation}
\label{e-t5}
\limsup_{i\to\infty}\frac{1}{n_i}\int
g_{n_i}(x)\;d\nu_{n_i}(x)\leq \mathcal G_*(\mu).
\end{equation}
Next we show further that $\mu$ is supported on $X$.  For this  we adopt some arguments from the proof of \cite[Theorem 1.1]{KenyonPeres96}. Notice that for each $i$, $\mu^{(i)}$ is $\sigma$-invariant supported on
$$
\bigcup_{k=0}^{n_i-1}\sigma^k X^{(n_i)}=\bigcup_{k=1}^{n_i}\sigma^{n_i-k} X^{(n_i)}.
$$
Hence $\mu$ is supported on
$$
\bigcap_{N=1}^\infty \overline{ \bigcup_{i=N}^\infty \bigcup_{k=1}^{n_i} \sigma^{n_i-k} X^{(n_i)}}.
$$
If $x$ is in this set, then for each $N\geq 1$ there exist integers $i(N)\geq N$ and $k(N)\in [1, n_{i(N)}]$ for which $d(x, \sigma^{n_{i(N)}-k(N)}X^{(n_{i(N)})})<1/N$, hence
\begin{equation}
\label{e-t6}
d(x, X)< \frac{1}{N}+ \sup_{y\in \sigma^{n_{i(N)}-k(N)}X^{(n_{i(N)})} } d(y, X)  \leq \frac{1}{N}+2^{-k(N)}
\end{equation}
and
\begin{equation}
\label{e-t7}
d(\sigma^{k(N)}x, X)<\frac{2^{k(N)}}{N}+\sup_{y\in \sigma^{n_{i(N)}-k(N)}X^{(n_{i(N)})} } d(\sigma^{k(N)}y, X)\leq \frac{2^{k(N)}}{N}+2^{-n_{i(N)}}.
\end{equation}
If the values $k(N)$ are unbounded as $N\to \infty$, then \eqref{e-t6} yields $x\in X$ while if they are bounded then some value of $k$ recurs infinitely often as $k(N)$ which implies that $\sigma^kx\in X$ by \eqref{e-t7}. Thus $\mu$ is supported on
$$
\bigcup_{k=0}^\infty\sigma^{-k}X.
$$
Since $\sigma X\subset X$, the set $(\sigma^{-1} X)\backslash X $ is wandering under $\sigma^{-1}$ (i.e. its preimages under powers of $\sigma$ are disjoint), so it must have zero $\mu$-measure. Consequently, $\mu\in \mathcal M(X,\sigma)$.

Notice that $h_{\mu^{(i)}}(\sigma)=\frac{1}{n_i}h_{\nu_{n_i}}(\sigma^{n_i})$ (see Lemma \ref{lem-A1}). By the upper semi-continuity of the entropy map,
$$
h_\mu(\sigma)\geq \limsup_{i\to \infty}h_{\mu^{(i)}}(\sigma)=\limsup_{i\to \infty}\frac{1}{n_i}h_{\nu_{n_i}}(\sigma^{n_i}),
$$
which, together with \eqref{e-t5}, yields that
\begin{eqnarray*}
h_\mu(\sigma)+\mathcal G_*(\mu)&\geq& \limsup_{i\to \infty}\frac{1}{n_i} \left(h_{\nu_{n_i}}(\sigma^{n_i})+\int g_{n_i} \; d\nu_{n_i}\right)\\
&=&\limsup_{n\to \infty}\frac{1}{n}P\left(X^{(n)},\sigma^n, g_n\right).
\end{eqnarray*}
Applying Theorem \ref{thm:sub-additive-VP}, we obtain \eqref{e-t2}. This completes the proof of the proposition.
\end{proof}

Next we present another auxiliary  result.

\begin{pro}
\label{pro-2.4}
Let $(X,\sigma)$ be a one-sided subshift over a finite alphabet ${\mathcal A}$ and $g, h\in C(X)$. Assume in addition that $h(x)<0$ for all $x\in X$. Let
$$r_0=\sup_{x\in X}\exp(h(x)).$$
Set for $0<r<r_0$,
$$
\mathcal A_r:=\left\{i_1\ldots i_n\in X^*:\; \sup_{x\in [i_1\ldots i_n]\cap X}\exp(S_nh(x))<r\leq \sup_{y\in [i_1\ldots i_{n-1}]\cap X}\exp(S_{n-1}h(y))\right\},
$$
where $X^*$ is the collection  of finite words allowed in $X$ and $S_nh(x):=\sum_{k=0}^{n-1} h(\sigma^kx)$.
Then
\begin{equation}
\label{e-4''}
\lim_{r\to 0}\frac{\log \left( \sum_{I\in \mathcal A_r}\sup_{x\in [I]\cap X} \exp(S_{|I|}g(x)) \right)}{\log r}=-t,
\end{equation}
where   $t$ is the unique real number so that $P(X,\sigma, g+th)=0$, and $|I|$ stands for the length of $I$.
\end{pro}

To prove the above result, we need the following.
\begin{lem}
\label{lem-simple}
Let $(X,\sigma)$ be a one-sided subshift over a finite alphabet ${\mathcal A}$  and $f\in C(X)$.
Then
$$
\lim_{n\to \infty} \frac{1}{n} \sup\{|S_nf(x)-S_nf(y)|:\; x_i=y_i \mbox{ for all } 1\leq i\leq n\}=0.
$$
\end{lem}
\begin{proof} The result is well-known. For the reader's convenience, we include a proof.

Define for $n\in \N$,
$$\mbox{var}_nf=\sup\{|f(x)-f(y)|:\; x_i=y_i \mbox{ for all } 1\leq i\leq n\}.$$
Since $f$ is uniformly continuous, $\mbox{var}_nf\to 0$ as $n\to \infty$. It follows that $$\lim_{n\to \infty} \frac{1}{n}\sum_{i=1}^n \mbox{var}_if=0.$$
This concludes the result of the lemma since $$\sup\{|S_nf(x)-S_nf(y)|:\; x_i=y_i \mbox{ for all } 1\leq i\leq n\}$$ is bounded above by $\sum_{i=1}^n \mbox{var}_if$.
\end{proof}
\begin{proof}[Proof of Proposition \ref{pro-2.4}] Set
$$
\Theta_r=\sum_{I\in \mathcal A_r}\sup_{x\in [I]\cap X} \exp(S_{|I|}g(x)),\quad r\in (0, r_0).
$$
Let $\epsilon>0$. It is enough to show that
\begin{equation}
\label{e-eq2}
r^{-t+\epsilon}\leq \Theta_r\leq r^{-t-\epsilon}
\end{equation} for sufficiently small $r$.

To this end, set for $0<r<r_0$,
$$
m(r)=\min\{|I|:\; I\in \mathcal A_r\}, \quad M(r)=\max\{|I|:\; I\in \mathcal A_r\}.
$$
From the definition of $\mathcal A_r$ and the negativity of $h$, it follows that there exist two positive constants $a, b$ such that
\begin{equation}
\label{e-eq1}
a\log(1/r)\leq m(r) \leq M(r)\leq b \log (1/r),\quad \forall r\in (0, r_0).
\end{equation}
Define
$$
\Gamma_r=\sum_{I\in \mathcal A_r}\sup_{x\in [I]\cap X} \exp(S_{|I|}(g+th)(x)),\quad r\in (0, r_0).
$$
By Lemma  \ref{lem-simple} and \eqref{e-eq1},  it is readily checked that
\begin{equation}\label{e-eq3}
 r^{t+\epsilon/2} \Theta_r\leq \Gamma_r\leq  r^{t-\epsilon/2} \Theta_r \; \mbox{ for sufficiently small }r,
\end{equation}
Hence to prove \eqref{e-eq2}, it suffices to prove that $r^{\epsilon/2}\leq  \Gamma_r\leq r^{-\epsilon/2}$ for small $r$.

We first prove $\Gamma_r>r^{\epsilon/2}$ when $r$ is small. Suppose on the contrary this is not true. Then by \eqref{e-eq1} we can find some $r>0$ and $\lambda>0$ such that $Z(r, \lambda)<1$, where
\begin{equation}
\label{e-eq4}
Z(r, \lambda):=\sum_{I\in \mathcal A_r}\exp(\lambda |I|) \sup_{x\in [I]\cap X} \exp(S_{|I|}(g+th)(x)).
\end{equation}
From Lemma 2.14 of \cite{Bowen08} it follows that $P(X, \sigma,  g+th)\leq -\lambda$, contradicting the fact that $P(X, \sigma,  g+th)=0$.  Hence we have $\Gamma_r>r^{\epsilon/2}$ when $r$ is sufficiently small.

Next we prove the inequality $\Gamma_r\leq r^{-\epsilon/2}$ for small $r$. To do this, fix $\lambda\in (0, \epsilon/(2b))$, where $b$ is the constant in \eqref{e-eq1}. We claim that there exists $0<r_1<r_0$ such that
\begin{equation}
\label{e-eq5}
Z(r, -\lambda)< 1 \mbox{ for all } r\in (0, r_1),
\end{equation}
where $Z$ is defined as in \eqref{e-eq4}. Since   $\lambda\in (0, \epsilon/(2b))$, it follows from \eqref{e-eq1} that $\exp(-\lambda|I|)\geq r^{\epsilon / 2}$ for any $I\in \mathcal A_r$. Hence \eqref{e-eq5} implies that $\Gamma_r\leq r^{-\epsilon/2}$ for $0<r<r_1$.

Now it remains to prove  \eqref{e-eq5}. Since  $P(X, \sigma,  g+th)=0$, by definition we have
$$
\lim_{n\to \infty} \frac{1}{n}\log \left(\sum_{I\in X^*: \; |I|=n} \sup_{x\in [I]\cap X} \exp(S_{|I|}(g+th)(x))\right)=0.
$$
Hence there exists a large $N$ such that $e^{-\lambda N/2}<1-e^{-\lambda/2}$ and for any $n>N$,
$$
 \gamma_n:=\sum_{I\in X^*: \; |I|=n} \exp(-\lambda |I|) \sup_{x\in [I]\cap X} \exp(S_{|I|}(g+th)(x))\leq \exp(-\lambda n/2).
$$
Take a small $r_1\in (0, r_0)$ so that $m(r)\geq N$ for any $0<r<r_1$. Hence by \eqref{e-eq1}, for any $0<r<r_1$, $\mathcal A_r\subset \{I\in X^*:\;|I|\geq N\}$ and so
$$
Z(r, -\lambda)\leq \sum_{n=N}^\infty \gamma_n\leq \sum_{n=N}^\infty \exp(-\lambda n/2)=\frac{e^{-\lambda N/2}}{1-e^{-\lambda/2}}<1.
$$
This proves \eqref{e-eq5} and we are done.
\end{proof}

\section{The proof of Theorem \ref{thm-1}}
\label{S-4}

Recall that for $T\in \R^{d\times d}$, $\alpha_1(T)\geq \cdots\geq \alpha_d(T)$ are the singular values of $T$, and $\phi^s(T)$ ($s\geq 0$) is defined as in \eqref{e-singular}.      We begin with an elementary but important lemma.
\begin{lem}
\label{lem-1}
Let $E\subset U\subset \R^d$, where $E$ is compact and $U$ is open. Let $k\in \{0,1,\ldots, d-1\}$. Then for any non-degenerate $C^1$ map $f:\; U\to \R^d$, there exists $r_0>0$ so that for any $y\in E$, $z\in B(y, r_0)$ and $0<r<r_0$,  the set $f(B(z, r))$ can be covered by
$$
(2d)^d\cdot\frac{\phi^k(D_yf)}{(\alpha_{k+1}(D_yf))^k}
$$
many balls of radius $\alpha_{k+1}(D_yf)r$.
\end{lem}
\begin{proof}
The result was implicitly proved in \cite[Lemma 3]{Zhang97} by using an idea of \cite{DouadyOesterle80}.
For the convenience of  the reader, we provide a detailed proof.

Set
$$
\gamma=\min_{y\in E}\alpha_d(D_yf).
$$
Then
\begin{equation}
\label{e-1-3}
B(0, \gamma)\leq  D_yf(B(0,1)) \mbox{ for all }y\in E.
\end{equation}
Since $f$ is $C^1$, non-degenerate on $U$ and $E$ is compact, it follows that $\gamma>0$.   Take $\epsilon=(\sqrt{d}-1)/2$. Then there exists a small $r_0>0$ such that for $u, v, w\in V_{2r_0}(E):=\{x: d(x, E)<2r_0\}$,
\begin{equation}
\label{e-1-1}
|f(u)-f(v)-D_vf(u-v)|\leq \epsilon \gamma |u-v| \quad \mbox{ if }|u-v|\leq r_0,
\end{equation}
and
\begin{equation}
\label{e-1-2}
D_vf(B(0,1))\subset ((1+\epsilon)D_wf)(B(0,1))\quad \mbox{ if }|v-w|\leq r_0.
\end{equation}

Now let $y\in E$ and $z\in B(y, r_0)$. For any $0<r<r_0$ and $x\in B(z, r)$,  taking $u=x$ and $v=z$ in \eqref{e-1-1}  gives
\begin{eqnarray*}
f(x)-f(z)-D_zf(x-z)\in B(0, \epsilon \gamma r),
\end{eqnarray*}
so by \eqref{e-1-2} and \eqref{e-1-3},
\begin{eqnarray*}
f(x)-f(z) &\in & D_zf(B(0,r))+B(0, \epsilon\gamma r)\\
&\mbox{}&\quad\subset  ((1+\epsilon) D_yf)(B(0,r))+B(0, \epsilon \gamma r) \qquad (\mbox{by \eqref{e-1-2}})\\
 &\mbox{}&\quad\subset   D_yf\left(B(0,(1+\epsilon)r))+ B(0,\epsilon r)\right) \;\qquad (\mbox{by \eqref{e-1-3}})\\
 &\mbox{}&\quad\subset   D_yf(B(0,(1+2\epsilon)r))\\
 &\mbox{}&\quad = D_yf\left(B\left(0,\sqrt{d}r\right)\right),
\end{eqnarray*}
where $A+A':=\{u+v:\; u\in A, \; v\in A'\}$.    Therefore  $$f(B(z, r))\subset f(z)+ D_yf\left(B\left(0,\sqrt{d}r\right)\right).$$
That is, $f(B(z, r))$ is contained in an ellipsoid which has principle axes of lengths  $2\sqrt{d}\alpha_i(D_yf)r$, $i=1,\ldots, d$.  Hence  $f(B(z, r))$ is contained in a rectangular parallelepiped of side lengths  $2\sqrt{d}\alpha_i(D_yf)r$, $i=1,\ldots, d$.   Now we can divide
such a parallelepiped into at most
\begin{eqnarray*}
\left(\prod_{i=1}^{k+1}\frac{2d\alpha_i(D_yf)}{\alpha_{k+1}(D_yf)}\right)\cdot (2d)^{d-k-1}
 \leq  (2d)^d \cdot\frac{\phi^k(D_yf)}{(\alpha_{k+1}(D_yf))^k}
\end{eqnarray*}
cubes of side $\frac{2}{\sqrt{d}}\cdot\alpha_{k+1}(D_yf)r$. Therefore this parallelepiped (and $f(B(z, r))$ as well) can be covered by $$(2d)^d \cdot\frac{\phi^k(D_yf)}{(\alpha_{k+1}(D_yf))^k}$$ balls of radius $\alpha_{k+1}(D_yf)r$.
\end{proof}

In the remaining part of this section, let $\{f_i\}_{i=1}^\ell$ be a $C^1$ IFS on $\R^d$ with attractor $K$. Let $(\Sigma,\sigma)$ be the one-sided full shift over the alphabet $\{1,\ldots, \ell\}$ and  $\Pi: \Sigma\to K$ the canonical coding map associated with the IFS (cf.~\eqref{e-coding}). As a consequence of  Lemma \ref{lem-1}, we obtain the following.

\begin{pro}
\label{pro-2}
Let $k\in \{0,1,\ldots, d-1\}$. Set $C=(2d)^d$. Then there exists $C_1>0$ such that for  $\mathbf{i}=(i_p)_{p=1}^\infty\in \Sigma$ and  $n\in \N$, the set $f_{\mathbf{i}|n}(K)$ can be covered by
 $C_1\prod_{p=0}^{n-1}G(\sigma^p \mathbf{i})$  balls of radius $\prod_{p=0}^{n-1}H(\sigma^p \mathbf{i})$,   where
\begin{equation}
\label{e-t9}
G(\mathbf{i}):=\frac{C\phi^k(D_{\Pi \sigma \mathbf{i}}f_{i_1})}{\alpha_{k+1}(D_{\Pi \sigma \mathbf{i}}f_{i_1})^k},\qquad H(\mathbf{i}):=\alpha_{k+1}(D_{\Pi \sigma \mathbf{i}}f_{i_1}).
\end{equation}
\end{pro}
\begin{proof}
Since $\{f_i\}_{i=1}^\ell$ is a $C^1$ IFS, there exists an open set $U\supset K$ such that each $f_i$ extends to a $C^1$ diffeomorphism $f_i:\; U\to f_i(U)$.  Applying Lemma \ref{lem-1} to the mappings $f_i$, we see that there exists $r_0>0$ such that for any $y\in K$, $z\in B(y, r_0)$, $0<r<r_0$ and $i\in \{1,\ldots, \ell\}$, the set $f_i(B(z, r))$ can be covered by
$$\theta(y,i):=\frac{C\phi^k(D_{y}f_{i})}{\alpha_{k+1}(D_{y}f_{i})^k}$$
balls of radius $\alpha_{k+1}(D_{y}f_{i})r$.

Since $f_1,\ldots, f_\ell$ are contracting on $U$, there exists $\gamma\in (0,1)$ such that
$$
|f_i(x)-f_i(y)|\leq \gamma|x-y|\quad \mbox{ for all }x, y\in U,\; i\in \{1,\ldots, \ell\}.
$$
It implies that
$\alpha_{1}(D_yf_i)\leq \gamma$ for any $y\in K$ and $i\in \{1,\ldots, \ell\}$. Take a large integer $n_0$ so that $$\gamma^{n_0}\max\{1,{\rm diam}(K)\}<r_0/2.$$
 Clearly there exists a large number $C_1$ so that the conclusion of the proposition holds for any positive integer  $n\leq n_0$ and $\mathbf{i}\in \Sigma$, i.e.~the set $f_{\mathbf{i}|n}(K)$ can be covered by
 $C_1\prod_{p=0}^{n-1}G(\sigma^p \mathbf{i})$  balls of radius $\prod_{p=0}^{n-1}H(\sigma^p \mathbf{i})$.   Below we show by induction  that this holds for all $n\in \N$ and $\mathbf{i}\in \Sigma$.

Suppose  for some $m\geq n_0$, the conclusion of the proposition holds for any positive integer  $n\leq m$ and $\mathbf{i}\in \Sigma$. Then for ${\bf i}\in \Sigma$,
$f|_{(\sigma {\bf i})|m}(K)$ can be covered by $C_1\prod_{p=0}^{m-1}G(\sigma^{p+1} \mathbf{i})$  balls of radius $\prod_{p=0}^{m-1}H(\sigma^{p+1} \mathbf{i})$. Let $B_1,\ldots B_N$ denote these balls. We may assume that $B_j\cap f|_{(\sigma {\bf i})|m}(K)\neq \emptyset$ for each $j$. Since  $$\prod_{p=0}^{m-1}H(\sigma^{p+1} \mathbf{i})\leq \gamma^m\leq \gamma^{n_0}<r_0/2$$
and $$d(\Pi\sigma\mathbf{i}, B_j\cap f_{(\sigma {\bf i})|m}(K))\leq {\rm  diam}(f_{(\sigma {\bf i})|m}(K))\leq \gamma^{n_0}{\rm  diam}(K)<r_0/2,$$
so the center of $B_j$ is in $B(\Pi \sigma {\bf i}, r_0)$.   Therefore $f_{i_1}(B_j)$ can be covered by $\theta(\Pi\sigma\mathbf{i}, i_1)=G(\bf{i})$ balls of radius $$H({\bf i})\cdot \prod_{p=0}^{m-1}H(\sigma^{p+1} \mathbf{i})=\prod_{p=0}^{m}H(\sigma^{p} \mathbf{i}).$$  Since  $f_{\mathbf {i}|{(m+1)}}(K)\subset \bigcup_{j=1}^N f_{i_1}(B_j)$, it follows that $f_{\mathbf {i}|{(m+1)}}(K)$ can be covered by
 $$G({\bf i})N\leq G({\bf i}) \cdot C_1\prod_{p=0}^{m-1}G(\sigma^{p+1} {\bf i})=C_1\prod_{p=0}^{m}G(\sigma^{p} {\bf i})$$
 balls of radius $\prod_{p=0}^{m}H(\sigma^{p} \mathbf{i})$.  Thus the proposition also holds for  $n=m+1$ and all ${\bf i}\in \Sigma$, as desired.
\end{proof}

Next we provide an upper bound of the upper box-counting dimension of the attractor $K$ of the IFS $\{f_i\}_{i=1}^\ell$.

\begin{pro}
\label{pro-4}
 Let $k\in \{0, 1,\ldots, d-1\}$.  Let $G, H:\Sigma\to \R$ be defined as in \eqref{e-t9}. Let $t$ be the unique real number so that $$P(\Sigma,\sigma,(\log G)+ t(\log H))=0.$$ Then
$
\overline{\dim}_BK\leq t.
$
\end{pro}

\begin{proof}
For short, we write $g=\log G$ and $h=\log H$.
Define
$$
r_{\rm min}=\min_{x\in \Sigma}\alpha_{k+1}(D_{\Pi\sigma x}f_{x_1}),\qquad r_{\rm max}=\max_{x\in \Sigma}\alpha_{k+1}(D_{\Pi\sigma x}f_{x_1}).
$$
Then $0<r_{\rm \min} \leq r_{\rm max}<1$. For $0<r<r_{\rm min}$, define
\begin{equation}
\label{e-t12}
\mathcal A_r=\left\{i_1\ldots i_n\in \Sigma^*:\; \sup_{x\in [i_1\ldots i_n]}S_nh(x)<\log r\leq \sup_{y\in [i_1\ldots i_{n-1}]}S_{n-1}h(y)\right\};
\end{equation}
clearly $\{[I]:\; I\in \mathcal A_r\}$ is a partition of $\Sigma$. By Proposition \ref{pro-2}, there exists a constant $C_1>0$ such that for each $0<r<r_{\rm min}$,   every $I\in \mathcal A_r$ and $x\in [I]$, $f_I(K)$ can be covered by
$$
C_1\exp(S_{|I|}g(x)) \leq C_1\exp\left(\sup_{y\in [I]}S_{|I|}g(y)\right)
$$
many balls of radius $$\exp(S_{|I|}h(x))\leq \exp\left(\sup_{y\in [I]}S_{|I|}h(y)\right)<r.$$
It follows that $K$ can be covered by
$$
C_1\sum_{I\in \mathcal A_r}\exp\left(\sup_{y\in [I]}S_{|I|}g(y)\right)
$$
many balls of radius $r$. Hence by Proposition \ref{pro-2.4},
$$
\overline{\dim}_BK\leq \limsup_{r\to 0}\frac{\log\left(\sum_{I\in \mathcal A_r}\exp\left(\sup_{y\in [I]}S_{|I|}g(y)\right)\right)}{\log (1/r)}= t.
$$
This completes the proof of the proposition.
\end{proof}

As an application of Proposition  \ref{pro-4}, we may estimate the upper box-counting dimension of the projections of a class of $\sigma$-invariant sets under the coding map.
\begin{pro}
\label{pro-upper}
Let $X$ be a compact  subset of $\Sigma$ satisfying $\sigma X\subset X$, and  $k\in \{0,\ldots, d-1\}$.
Then for each $n\in \N$,
$$
\overline{\dim}_B\Pi\left(X^{(n)}\right)\leq t_n,
$$
where $X^{(n)}$ is defined as in \eqref{e-xn},   $t_n$ is the unique number for which $$P\left(X^{(n)}, \sigma^n, (\log G_n)+ t_n(\log H_n)\right)=0,$$ and $G_n, H_n$ are continuous functions on $\Sigma$ defined by
\begin{equation}
\label{e-10}G_n(y):=\frac{C\phi^k(D_{\Pi \sigma^n y}f_{y|n})}{\alpha_{k+1}(D_{\Pi \sigma^n y}f_{y|n})^k},\qquad H_n(y):=\alpha_{k+1}(D_{\Pi \sigma^n y}f_{y|n}),
\end{equation}
with  $C=(2d)^d$.
\end{pro}

\begin{proof}
The result is obtained by applying Proposition \ref{pro-4} to the IFS $\{f_{I}:\; I\in X^*_n\}$ instead of $\{f_i\}_{i=1}^\ell$, where $X_n^*$ stands for the collection of words of length $n$ allowed in $X$.
\end{proof}

Now we are ready to prove Theorem \ref{thm-1}.

\begin{proof}[Proof of Theorem \ref{thm-1}] Write $s=\dim_SX$. We may assume that $s<d$; otherwise we have nothing left to prove. Set $k=[s]$, i.e. $k$ is the largest integer less than or equal to $s$.
Let $\mathcal U=\{u_n\}_{n=1}^\infty$ be the sub-additive potential on $\Sigma$ defined by
$$u_n(x)=\log \phi^s(D_{\Pi\sigma^n x} f_{x|n}).$$
Then $P(X, \sigma, {\mathcal U})=0$ by the definition of $\dim_SX$.  Hence by Proposition \ref{pro-1},
 $$
 \lim_{n\to \infty}\frac{1}{n}P\left(X^{(n)}, \sigma^n, u_n\right)=\inf_{n\geq 1}\frac{1}{n}P\left(X^{(n)}, \sigma^n, u_n\right)=0.
 $$
 It follows that for each $\epsilon>0$, there exists $N_\epsilon>0$ such that
 \begin{equation}
 \label{e-t8}
 0\leq P\left(X^{(n)}, \sigma^n, u_n\right)\leq n\epsilon \quad\mbox{ for  }\; n>N_\epsilon.
 \end{equation}

For $n\in {\Bbb N}$, let $s_n$ be the unique real number so that $P(X^{(n)}, \sigma^n, v_n)=0$, where $v_n$ is a continuous function on $\Sigma$ defined by
 \begin{eqnarray*}
 v_n(x)&=& \log G_n(x)+ s_n \log H_n(x)\\
 &=&\log \left( C\phi^k(D_{\Pi\sigma^nx}f_{x|n})\alpha_{k+1}(D_{\Pi\sigma^nx}f_{x|n})^{s_n-k}\right),
 \end{eqnarray*}
where $G_n, H_n$ are defined in \eqref{e-10} and $C=(2d)^d$. By Proposition \ref{pro-upper}, $$\overline{\dim}_B\Pi(X)\leq \overline{\dim}_B\Pi\left(X^{(n)}\right)\leq  s_n.$$
 If $s\geq s_n$ for some $n$, then $\overline{\dim}_B\Pi(X)\leq s_n\leq s$ and we are done. In what follows, we assume that $s<s_n$ for each $n$. Then for each $x\in \Sigma$,
  \begin{eqnarray*}
  u_n(x)-v_n(x)&=& -\log C +(s-s_n)\log \alpha_{k+1}(D_{\Pi\sigma^nx}f_{x|n})\\
  &\geq & -\log C +(s-s_n)\log \alpha_1(D_{\Pi\sigma^nx}f_{x|n})\\
  &\geq & -\log C+ n(s-s_n) \log \theta,
  \end{eqnarray*}
 where
$$
 \theta:=\max_{y\in \Sigma} \alpha_1(D_{\Pi\sigma y}f_{y_1})<1.
 $$
 Hence
 \begin{align*}
  P(X^{(n)}, \sigma^n, u_n)&=P(X^{(n)}, \sigma^n, u_n)-P(X^{(n)}, \sigma^n, v_n)\\
 & \geq \inf_{x\in \Sigma} (u_n(x)-v_n(x))\\
  &\geq -\log C+ n(s-s_n) \log \theta,
 \end{align*}
 where in the second inequality, we used \cite[Theorem 9.7(iv)]{Walters82}.
 Combining it with \eqref{e-t8} yields that for $n\geq N_\epsilon$,
 $$
 n\epsilon\geq -\log  C + n(s-s_n) \log \theta,
 $$
 so
 $$
 s\geq s_n + \frac{\epsilon+ n^{-1} \log C}{\log \theta}\geq \overline{\dim}_B\Pi(X)+\frac{\epsilon+ n^{-1} \log C}{\log \theta}.
 $$
 Letting $n\to \infty$ and then $\epsilon\to 0$, we obtain $s\geq \overline{\dim}_B\Pi(X)$, as desired.
\end{proof}

\section{The proof of Theorem \ref{thm-2}}
\label{S-5}

Let $\Pi:\Sigma\to \R^d$ be the coding map associated with a $C^1$ IFS $\{f_i\}_{i=1}^\ell$ on $\R^d$ (cf.~ \eqref{e-coding}).
 For $E\subset \R^d$ and $\delta>0$, let $N_\delta(E)$ denote the smallest integer $N$ for which $E$ can  be covered by $N$ closed balls of radius $\delta$.
 For $T\in \R^{d\times d}$,  let $\alpha_1(T)\geq\cdots\geq \alpha_d(T)$ denote the  singular values of $T$, and let $\phi^s(T)$ be the singular value function defined as in \eqref{e-singular}.

 The following geometric counting lemma plays an important role in the proof of Theorem \ref{thm-2}. It is  of independent interest as well.
\begin{lem}
\label{lem-4.1}
Let $m$ be an ergodic $\sigma$-invariant Borel probability measure on $\Sigma$.  Set
\begin{equation}
\label{e-lambda}
\lambda_i:=\lim_{n\to \infty} \frac{1}{n}\int \log\left( \alpha_i(D_{\Pi\sigma^nx}f_{x|n})\right)\; dm(x),\quad i=1,\ldots, d.
\end{equation}
Let $k\in \{0,\ldots, d-1\}$.  Write $u:=\exp(\lambda_{k+1})$. Then for $m$-a.e.~$x\in \Sigma$,
\begin{equation}
\label{e-gamma1}
\limsup_{n\to \infty} \frac{1}{n}\log N_{u^n}(\Pi([x|n]))\leq (\lambda_1+\cdots+\lambda_k)-k\lambda_{k+1}.
\end{equation}

\end{lem}

\begin{proof}
 It is known (see, e.g.~\cite[Theorem 3.3.3]{Arnold98}) that for $m$-a.e.~$x$,
 $$
\lim_{p\to \infty}\frac{1}{p} \log  \left(\alpha_i (D_{\Pi \sigma^p x}f_{x|p})\right)=\lambda_i,\qquad i=1,\ldots, d,
 $$
 and
 $$
\lim_{p\to \infty}\frac{1}{p} \log  \left(\phi^i (D_{\Pi \sigma^p x}f_{x|p})\right)=\lambda_1+\cdots+\lambda_i,\qquad i=1,\ldots, d.
 $$

  For $i\in \{1,\ldots, d\}$, $x\in \Sigma$ and $p\in {\Bbb N}$, set
  \begin{equation*}
  \begin{split}
  w_p^{(i)}(x)&= \log \phi^i(D_{\Pi \sigma^p x}f_{x|p}), \\
  v_p^{(i)}(x)&= \log \alpha_i(D_{\Pi \sigma^p x}f_{x|p}).
  \end{split}
  \end{equation*}
 Then by the definition of $\phi^i$,
 \begin{equation}\label{e-w1}
 w_p^{(i)}=v_p^{(1)}+\cdots+v_p^{(i)}.
 \end{equation}
 Since $\phi^i$ is sub-multiplicative on $\R^{d\times d}$ (cf. \cite[Lemma 2.1]{Falconer88}), $\left\{w^{(i)}_p\right\}_{p=1}^\infty$   is a sub-additive potential  satisfying
  $$
 |w_p^{(i)}(x)|\leq pC,\quad p\in {\Bbb N}, \;x\in \Sigma,
  $$
  for some constant $C>0$.  Set ${\mathcal C}_p:=\{B\in {\mathcal B}(\Sigma):\; \sigma^{-p}B=B\; a.e.\}$. Then by Lemma \ref{lem-2.3.1}, for $m$-a.e.~$x$,
  \begin{equation}
  \label{e-w2}
  \lim_{p\to \infty}E\left(\frac{w^{(i)}_p}{p}  \Big| {\mathcal C}_p \right)(x) =\lim_{p\to \infty} \frac{1}{p}w^{(i)}_p(x)=\lambda_1+\cdots+\lambda_i, \quad i=1,\ldots, d,
  \end{equation}
  and so by \eqref{e-w1},
  \begin{equation}
  \label{e-w3}
  \lim_{p\to \infty}E\left(\frac{v^{(i)}_p}{p}  \Big| {\mathcal C}_p \right)(x) =\lambda_i, \quad i=1,\ldots, d,
  \end{equation}

  Let $p\in \N$. Applying Proposition \ref{pro-2} to the IFS $\{f_I:\; I\in \mathcal A^p\}$, we see that there exists a  positive number $C_1(p)$ such that for any $x\in \Sigma$ and $n\in \N$, the set $\Pi([x|np])=f_{x|np}(K)$ can be covered by $C_1(p)\prod_{i=0}^{n-1}G_p(\sigma^{pi} x)$  balls of radius $\prod_{i=0}^{n-1}H_p(\sigma^{ip} x)$,   where
\begin{equation}
\label{e-n1}
G_p(x):=\frac{C\phi^k(D_{\Pi \sigma^px }f_{x|p})}{\alpha_{k+1}(D_{\Pi \sigma^p x}f_{x|p})^k},\qquad H_p(x):=\alpha_{k+1}(D_{\Pi \sigma^px}f_{x|p}),
\end{equation}
with $C=(2d)^d$.
By the Birkhoff's ergodic theorem, for $m$-a.e.~$x\in \Sigma$,
\begin{equation}
\label{e-n2}
\lim_{n\to \infty}\frac{1}{np}\log \left(C_1(p)\prod_{i=0}^{n-1}G_p(\sigma^{pi} x)\right)=\frac{\log C}{p}+E\left(\frac{w^{(k)}_p}{p}  \Big| {\mathcal C}_p \right)(x)-kE\left(\frac{v^{(k+1)}_p}{p}  \Big| {\mathcal C}_p \right)(x)
\end{equation}
and
\begin{equation}
\label{e-n3}
\lim_{n\to \infty}\frac{1}{np}\log \left( \prod_{i=0}^{n-1}H_p(\sigma^{ip} x)\right)=E\left(\frac{v^{(k+1)}_p}{p} \Big| {\mathcal C}_p \right)(x)
,
\end{equation}

Let $\epsilon>0$. By \eqref{e-w2}, \eqref{e-w3}, \eqref{e-n2},\eqref{e-n3},  for $m$-a.e.~$x$ there exists a positive integer $p_0(x)$ such that for any $p\geq p_0(x)$,
\begin{equation}
\label{e-n4}
\prod_{i=0}^{n-1}H_p(\sigma^{ip} x)\leq (u+\epsilon)^{np}\quad \mbox{ for large enough }n,
\end{equation}
and
\begin{equation}
\label{e-n4'}
\frac{1}{np}\log \left( C_1(p)\prod_{i=0}^{n-1}G_p(\sigma^{pi} x)\right) \leq \lambda_1+\cdots +\lambda_k-k\lambda_{k+1}+\epsilon \quad \mbox{ for large enough }n,
\end{equation}

Fix such an $x$ and let $p\geq p_0(x)$. By \eqref{e-n4},
$$
N_{(u+\epsilon)^{np}}(\Pi([x|pn]))\leq  C_1(p)\prod_{i=0}^{n-1}G_p(\sigma^{pi} x) \quad \mbox{ for large enough }n.
$$
Notice that there exists a constant $C_2=C_2(d)>0$ such that a ball of radius $(u+\epsilon)^{np}$ in $\R^d$ can be covered by $C_2 (1+\epsilon/u)^{dnp}$ balls of radius $u^{np}$. It follows that  for large enough $n$,
\begin{eqnarray*}
N_{u^{np}}(\Pi([x|pn]))&\leq & C_2 (1+\epsilon/u)^{dnp} N_{(u+\epsilon)^{np}}(\Pi([x|pn]))\\
&\leq &
  C_1(p) C_2 (1+\epsilon/u)^{dnp} \prod_{i=0}^{n-1}G_p(\sigma^{pi} x).
\end{eqnarray*}
Hence by \eqref{e-n4'},
\begin{eqnarray*}
\limsup_{n\to \infty}\frac{1}{n}\log N_{u^{n}}(\Pi([x|n]))&=&\limsup_{n\to \infty} \frac{1}{np}\log N_{u^{np}}(\Pi([x|pn]))\\
&\leq & d\log(1+\epsilon/u)+ \limsup_{n\to \infty} \frac{1}{np}\log \left(\prod_{i=0}^{n-1}G_p(\sigma^{pi} x)\right)\\
&\leq  & d\log(1+\epsilon/u)+\epsilon+\lambda_1+\cdots+\lambda_k-k\lambda_{k+1},
\end{eqnarray*}
where the first equality follows from the fact that for $pn\leq m<p(n+1)$,
\begin{eqnarray*}
N_{u^{m}}(\Pi([x|m]))\leq N_{u^{m}}(\Pi([x|pn]))&\leq& 4^d (u^{pn-m})^d N_{u^{np}}(\Pi([x|pn]))\\
&\leq& 4^d u^{-pd} N_{u^{np}}(\Pi([x|pn])),
\end{eqnarray*}
using the fact that for $R>r>0$, a ball of radius $R$ in $\R^d$ can be covered by $(4R/r)^d$ balls of radius $r$.
Letting  $\epsilon\to 0$ yields the desired inequality \eqref{e-gamma1}.
\end{proof}

The following result is also needed in the proof of Theorem \ref{thm-2}.
\begin{lem}
\label{lem-est}
Let $m$ be a Borel probability measure on $\Sigma$. Let $\rho, \epsilon \in (0,1)$. Then for $m$-a.e.~$x=(x_n)_{n=1}^\infty\in \Sigma$,
\begin{equation}
\label{e-Borel}
m\circ\Pi^{-1} (B(\Pi x, 2\rho^n))\geq (1-\epsilon)^n\frac{m([x_1\ldots x_n])}{N_{\rho^n}(\Pi([x_1\ldots x_n]))} \quad \mbox{ for  large enough }n.
\end{equation}
 \end{lem}
\begin{proof}
The formulation and the proof of the above lemma are adapted from an argument given by Jordan \cite{Jordan11}. A similar idea was also employed  in the proof of \cite[Theorem 2.2]{Rossi14}.

For $n\in \N$, let $\Lambda_n$ denote the set of the points $x=(x_n)_{n=1}^\infty\in \Sigma$ such that
$$
m\circ\Pi^{-1} (B(\Pi x, 2\rho^n))< (1-\epsilon)^n\frac{m([x_1\ldots x_n])}{N_{\rho^n}(\Pi([x_1\ldots x_n]))}.
$$
To prove that \eqref{e-Borel} holds almost everywhere,  by the Borel-Cantelli lemma it suffices to show that
\begin{equation}
\label{e-eta}
\sum_{n=1}^\infty m(\Lambda_n)<\infty.
\end{equation}
For this purpose, let us  estimate $m(\Lambda_n)$.  Fix $n\in \N$ and $I\in \mathcal A^n$. Notice that $\Pi([I])$ can be covered by $N_{\rho^n}(\Pi([I]))$ balls of radius $\rho^n$. As a consequence, there exists $L\leq N_{\rho^n}(\Pi([I]))$ such that $\Pi(\Lambda_n\cap [I])$ can be covered by $L$ balls of radius $\rho^n$, say, $B_1,\ldots, B_L$. We may assume that $\Pi(\Lambda_n\cap [I])\cap B_i\neq \emptyset$ for each $1\leq i\leq L$. Hence for each $i$, we may pick $x^{(i)}\in \Lambda_n \cap [I]$ such that $\Pi x^{(i)} \in B_i$. Clearly $B_i\subset B(\Pi x^{(i)}, 2 \rho^n)$. Since $x^{(i)}\in \Lambda_n \cap [I]$, by the definition of $\Lambda_n$ we obtain
$$m\circ\Pi^{-1} (B(\Pi x^{(i)}, 2\rho^n))< (1-\epsilon)^n\frac{m([I])}{N_{\rho^n}(\Pi([I]))}.$$
It follows that
\begin{eqnarray*}
m(\Lambda_n\cap [I])&\leq& m\circ \Pi^{-1}(\Pi(\Lambda_n\cap [I]))\\
&\leq & m\circ \Pi^{-1}\left(\bigcup_{i=1}^LB_i\right)\\
&\leq & m\circ \Pi^{-1}\left(\bigcup_{i=1}^LB(\Pi x^{(i)}, 2\rho^n)\right)\\
&\leq & L (1-\epsilon)^n\frac{m([I])}{N_{\rho^n}(\Pi([I]))}\\
&\leq & (1-\epsilon)^nm([I]).
\end{eqnarray*}
Summing over $I\in \mathcal A^n$ yields that $m(\Lambda_n)\leq (1-\epsilon)^n$, which implies \eqref{e-eta}.
\end{proof}
\begin{rem}
{\rm
Lemma \ref{lem-est} remains valid when the coding map $\Pi:\Sigma\to \R^d$ is replaced by any Borel measuarble map from $\Sigma$ to $\R^d$.
}
\end{rem}

Now we are ready to prove Theorem \ref{thm-2}.

\begin{proof}[Proof of Theorem \ref{thm-2}]
We may assume that $s:=\dim_{L}m<d$; otherwise there is nothing left to prove. Set $k=[s]$. Let $\lambda_i$, $i=1,\ldots, d$, be defined as in \eqref{e-lambda}. Then by Definition \ref{de-2},
$$
h_m(\sigma)+\lambda_1+\cdots+\lambda_k+(s-k)\lambda_{k+1}=0.
$$
Let $u=\exp(\lambda_{k+1})$ and $\epsilon\in (0,1)$. Applying Lemma \ref{lem-est} yields that for $m$-a.e.~$x=(x_n)_{n=1}^\infty\in \Sigma$,
\begin{equation}
\label{e-Borel'}
m\circ\Pi^{-1} (B(\Pi x, 2u^n))\geq (1-\epsilon)^n\frac{m([x_1\ldots x_n])}{N_{u^n}(\Pi([x_1\ldots x_n]))} \quad \mbox{ for  large enough }n.
\end{equation}
It follows that for $m$-a.e.~$x=(x_n)_{n=1}^\infty\in \Sigma$,
\begin{eqnarray*}
&&\overline{d}(m\circ \Pi^{-1}, \Pi x)\\
&&\mbox{}\quad=\limsup_{n\to \infty} \frac{\log\left(m\circ\Pi^{-1} (B(\Pi x, 2u^n))\right)}{n\log u} \\
&&\mbox{}\quad\leq \frac{\log(1-\epsilon)}{\log u}+\limsup_{n\to \infty}\left(\frac{\log m([x_1\ldots x_n])}{n \log u} -\frac{\log N_{u^n}(\Pi([x_1\ldots x_n]))}{n\log u}\right)\\
&&\mbox{}\quad\leq \frac{\log(1-\epsilon)-h_m(\sigma)-(\lambda_1+\cdots+\lambda_k)+k\lambda_{k+1}}{\log u}\\
&&\mbox{}\quad= \frac{\log(1-\epsilon)+s\lambda_{k+1}}{\lambda_{k+1}},
\end{eqnarray*}
where in the third inequality, we used the Shannon-McMillan-Breiman theorem (cf.~\cite[Page.~93]{Walters82}) and Lemma \ref{lem-4.1} (keeping in mind that $\log u=\lambda_{k+1}<0$). Letting $\epsilon\to 0$ yields the desired result. \end{proof}

\section{Upper bound for the box-counting dimension of $C^1$-repellers and the Lyapunov dimensions of ergodic invariant measures}
\label{S-6}
 Throughout this section let ${\pmb M}$ be a smooth Riemannian manifold of dimension $d$ and $\psi: {\pmb M}\to {\pmb M}$ a $C^1$-map. Let $\Lambda$ be a compact subset of ${\pmb M}$ such that $\psi(\Lambda)=\Lambda$, and assume that $\Lambda$ is a repeller of $\psi$ (cf. Section \ref{S-1}). Below we first introduce the definitions of singularity and Lyapunov dimensions for the case of repellers, which are somehow sightly different from that for the case of IFSs.

\begin{de}
\label{de-6.1}
{\rm The {\it singularity dimension} of $\Lambda$ with respect to $\psi$, written as $\dim_{S^*}\Lambda$, is the unique real value $s$ for which
$$
P(\Lambda, \psi, {\mathcal G}^ s)=0,
$$
where ${\mathcal G}^s=\{g^s_n\}_{n=1}^\infty$ is the sub-additive potential on $\Lambda$ defined by
\begin{equation}
\label{e-gs}
g^s_n(z)=\log \phi^s((D_z \psi^n)^{-1}),\quad z\in \Lambda.
\end{equation}
}
\end{de}

\begin{de}
\label{de-6.2}
{\rm
For an ergodic $\psi$-invariant measure $\mu$ supported on $\Lambda$, the {\it Lyapunov  dimension} of $\mu$ with respect to $\psi$, written as $\dim_{L^*}\mu$, is the unique real value $s$ for which
$$
h_\mu(\psi)+ \mathcal G^s_*(\mu)=0,
$$
where ${\mathcal G}^s=\{g^s_n\}_{n=1}^\infty$ is defined as in \eqref{e-gs} and $\mathcal G^s_*(\mu)=\lim_{n\to \infty}\frac{1}{n} \int g^s_n d\mu$.
}
\end{de}

Before proving Theorems \ref{thm-3}-\ref{thm-4}, we recall some definitions and necessary facts about $C^1$ repellers.

A finite closed cover
 $\{R_1,\cdots,R_{\ell}\}$ of $\Lambda$  is called a {\it Markov partition} of
$\Lambda$ with respect to $\psi$ if:
\begin{itemize}
 \item[(i)] $\overline{\mbox{int}R_i}=R_i$ for each
 $i=1,\ldots,\ell$;

\item[(ii)]  $\mbox{int}R_i \cap \mbox{int}R_j=\emptyset$ for
$i\neq
 j$; and

\item[(iii)]
   each $\psi(R_i)$ is the union of a subfamily of
$\{R_j\}_{j=1}^{\ell}$.
\end{itemize}

 It is well-known that any repeller of an expanding map  has Markov partitions of arbitrary
small diameter (see \cite{Rue}, p.146).  Let $\{R_1,\ldots, R_\ell\}$ be a Markov partition of $\Lambda$ with respect to $\psi$. It is known that this dynamical system induces a subshift space of finite  type $(\Sigma_A,\sigma)$ over the alphabet $\{1,\ldots, \ell\}$,   where $A=(a_{ij})$ is the transfer
matrix of the Markov partition, namely, $a_{ij}=1$ if $\mbox{int}R_i\cap \psi^{-1}(\mbox{int}R_j)\neq \emptyset$ and
$a_{ij}=0$ otherwise \cite{Rue}.   This gives the coding map
$\Pi: \Sigma_A\rightarrow \Lambda$ such that
\begin{equation}
\label{e-pia}
\Pi({\bf i})=\bigcap_{n\geq 1}\psi^{-(n-1)}(R_{i_n}), \qquad \forall \
{\bf i}=(i_1i_2\cdots)\in \Sigma_A,
\end{equation}
and the following diagram
\begin{equation}\label{077}
  \xymatrix{ \Sigma_{A} \ar[r]^{\sigma}  \ar[d]_\Pi   & %
                   \Sigma_{A}\ar[d]^{\Pi}\\ %
\Lambda\ar[r]_{\psi} & \Lambda  }%
\end{equation}
commutes. (Keep in mind that throughout this section, $\Pi$ denotes the coding map for the repeller $\Lambda$ (no longer for the coding map for an IFS used in the previous sections.)

The coding map $\Pi$ is a H\"older continuous
surjection. Moreover there is a positive integer $q$ so that
$\Pi^{-1}(z)$ has at most $q$ elements for each $z\in \Lambda$ (see
\cite{Rue}, p.147).

For $n\geq 1$, define $$\Sigma_{A,n}:=\{i_1\ldots i_n\in \{1,\ldots, \ell\}^n:\; a_{i_ki_{k+1}}=1 \mbox{ for } 1\leq k\leq n-1\}.$$
For any  word $I=i_1\ldots i_n\in \Sigma_{A, n}$, the set $\bigcap_{k=1}^{n}\psi^{-(k-1)}(R_{i_{k}})$
 is called a {\it basic set} and is
denoted by $R_{I}$.

The proof of Theorem \ref{thm-3} is similar to that of Theorem \ref{thm-1}.  We begin with the following lemma, which is a slight variant of Lemma \ref{lem-1}.
\begin{lem}
\label{lem-6.5}
 Let $E\subset U\subset {\pmb M}$, where $E$ is compact and $U$ is open. Let $k\in \{0,1,\ldots, d-1\}$. Then for any non-degenerate $C^1$ map $f:\; U\to {\pmb M}$, there exists $r_0>0$ so that for any $y\in E$, $z\in B(y, r_0)$ and $0<r<r_0$,  the set $f(B(z, r))$ can be covered by
$$
C_{\pmb M}\cdot\frac{\phi^k(D_yf)}{(\alpha_{k+1}(D_yf))^k}
$$
many balls of radius $\alpha_{k+1}(D_yf)r$, where $C_{\pmb M}$ is a positive constant depending on ${\pmb M}$.
\end{lem}
\begin{proof}
This can be done by routinely  modifying the proof of Lemma \ref{lem-1} and using similar arguments of the proof of \cite[Corollary]{Zhang97}.
\end{proof}

Let $\delta>0$ be small enough so that $\psi: B(z, \delta)\to f(B(z, \delta))$ is a diffeomorphism for each $z$ in the $\delta$-neighborhood of $\Lambda$. Suppose that  $\{R_1,\ldots, R_\ell\}$ is a Markov partition of $\Lambda$ with diameter less than $\delta$.

The following result is an analogue of Proposition \ref{pro-2}.
\begin{pro}
\label{pro-5.2}
Let $k\in \{0,1,\ldots, d-1\}$. Set $C_{\pmb M}$ be the constant in Lemma \ref{lem-6.5}. Then there exists $C_1>0$ such that for  all $\mathbf{i}=(i_p)_{p=1}^\infty\in \Sigma_A$ and  $n\in \N$, the basic set $R_{\mathbf{i}|n}$ can be covered by
 $C_1\prod_{p=0}^{n-1}G(\sigma^p \mathbf{i})$  balls of radius $\prod_{p=0}^{n-1}H(\sigma^p \mathbf{i})$,   where
\begin{equation}
\label{e-t9'}
G(\mathbf{i}):=\frac{C_{\pmb M}\phi^k((D_{\Pi \mathbf{i}}\psi)^{-1})}{\alpha_{k+1}((D_{\Pi  \mathbf{i}}\psi)^{-1})^k},\qquad H(\mathbf{i}):=\alpha_{k+1}((D_{\Pi \mathbf{i}} \psi)^{-1}).
\end{equation}
\end{pro}

\begin{rem}
{\rm
The definitions of the functions $G$ and $H$ in the above proposition are slightly different from that in Proposition \ref{pro-2}.
}
\end{rem}

\begin{proof}[Proof of Proposition \ref{pro-5.2}] The proof is adapted from that of Proposition \ref{pro-2}. For the reader's convenience, we provide the full details.

First we construct a local inverse $f_{i,j}$ of $\psi$ for each pair $(i,j)$ with  $ij\in \Sigma_{A,2}$. To do so, notice that $\psi(R_{ij})=R_j$ for each $ij\in \Sigma_{A,2}$.   Since $\psi$ is a diffeomorphism restricted on a small neighborhood of $R_i$, we can find open sets $\widetilde{R}_{ij}$ and $\widetilde{R_j}$ such that  $\widetilde{R}_{ij} \supset R_{ij}$, $\widetilde{R_j}\supset R_j$, $\psi(\widetilde{R}_{ij})=\widetilde{R_j}$ and  $\psi: \widetilde{R}_{ij}\to \widetilde{R_j}$ is diffeomorphic.
Then we take $f_{i,j}: \widetilde{R_j}\to \widetilde{R}_{ij}$ to be the inverse of $\psi: \widetilde{R}_{ij}\to \widetilde{R_j}$, and the construction is done.

For any ${\bf i}=(i_n)_{n=1}^\infty \in \Sigma_A$,  we see that  $\Pi\sigma {\bf i}\in R_{i_2}\subset \widetilde{R_{i_2}}$ and $(\psi\circ f_{i_1, i_2})|_{\widetilde{R_{i_2}}}$ is the identity restricted on $\widetilde{R_{i_2}}$. Since $\psi(\Pi{\bf i})=\Pi\sigma{\bf i}$, it follows that $f_{i_1, i_2}(\Pi\sigma {\bf i})=\Pi{\bf i}$. Differentiating $\psi\circ f_{i_1, i_2}$ at $\Pi\sigma {\bf i}$ and applying the chain rule, we get
$$
(D_{\Pi {\bf i}} \psi)(D_{\Pi\sigma {\bf i}}f_{i_1, i_2})=\mbox{Identity},
$$
so
\begin{equation}\label{e-63}
D_{\Pi\sigma {\bf i}}f_{i_1, i_2}=(D_{\Pi {\bf i}} \psi)^{-1}.
\end{equation}

According to Lemma \ref{lem-6.5}, there exists $r_0>0$ such that for each $ij\in \Sigma_{A,2}$,  $y\in R_j$,  $z\in B(y, r_0)$ and $0<r<r_0$,  the set $f_{i,j}(B(z, r))$ can be covered by
$$
C_{\pmb M}\cdot\frac{\phi^k(D_yf_{i,j})}{(\alpha_{k+1}(D_yf_{i,j}))^k}
$$
many balls of radius $\alpha_{k+1}(D_yf_{i,j})r$.

Since $\psi$ is expanding on $\Lambda$, there exists $\gamma\in (0,1)$ such that
$\sup_{{\bf i}\in \Sigma_A}\alpha_1((D_{\Pi \mathbf{i}} \psi)^{-1})<\gamma$. Then
$\sup_{{\bf i}\in \Sigma_A}H({\bf i})<\gamma$ and
\begin{equation}\label{e-61}
\mbox{diam}(R_{{\bf i}|(n+1)})\leq \gamma \mbox{diam}(R_{(\sigma{\bf i})|n})
\end{equation}
for all ${\bf i}\in \Sigma_A$ and $n\in \N$. Take a large integer $n_0$ so that
\begin{equation}\label{e-62} \gamma^{n_0-1} \max\{1, \mbox{diam}(\Lambda)\}<r_0/2.
\end{equation}
By \eqref{e-61}-\eqref{e-62}, $\mbox{diam}(R_{{\bf i}|n})<r_0/2$ for all ${\bf i}\in \Sigma_A$ and $n\geq n_0$.

 Clearly there exists a large number $C_1$ so that the conclusion of the proposition holds for any positive integer  $n\leq n_0$ and $\mathbf{i}\in \Sigma_A$, i.e.~the set $R_{\mathbf{i}|n}$ can be covered by
 $C_1\prod_{p=0}^{n-1}G(\sigma^p \mathbf{i})$  balls of radius $\prod_{p=0}^{n-1}H(\sigma^p \mathbf{i})$.   Below we show by induction  that this holds for all $n\in \N$ and $\mathbf{i}\in \Sigma_A$.

Suppose  for some $m\geq n_0$, the conclusion of the proposition holds for any positive integer  $n\leq m$ and $\mathbf{i}\in \Sigma_A$. Then for given ${\bf i}=(i_n)_{n=1}^\infty\in \Sigma_A$,
$R_{(\sigma {\bf i})|m}$ can be covered by $C_1\prod_{p=0}^{m-1}G(\sigma^{p+1} \mathbf{i})$  balls of radius $\prod_{p=0}^{m-1}H(\sigma^{p+1} \mathbf{i})$. Let $B_1,\ldots, B_N$ denote these balls. We may assume that $B_j\cap R_{(\sigma {\bf i})|m}\neq \emptyset$ for each $j$. Since  $$\prod_{p=0}^{m-1}H(\sigma^{p+1} \mathbf{i})\leq \gamma^m\leq \gamma^{n_0}<r_0/2$$
and $$d(\Pi\sigma\mathbf{i}, B_j\cap R_{(\sigma {\bf i})|m})\leq {\rm  diam}(R_{(\sigma {\bf i})|m})<r_0/2,$$
so the center of $B_j$ is in $B(\Pi \sigma {\bf i}, r_0)$.   Therefore by Lemma \ref{lem-6.5} and \eqref{e-63},  $f_{i_1, i_2}(B_j)$ can be covered by
$$
C_{\pmb M}\cdot\frac{\phi^k(D_{\Pi \sigma {\bf i}}f_{i_1, i_2})}{(\alpha_{k+1}(D_{\Pi \sigma {\bf i}}f_{i_1, i_2}))^k}=G({\bf i})
$$
balls of radius $$\alpha_{k+1}(D_{\Pi\sigma{\bf i}}f_{i_1,i_2})\cdot\prod_{p=0}^{m-1}H(\sigma^{p+1} \mathbf{i})= H({\bf i})\cdot \prod_{p=0}^{m-1}H(\sigma^{p+1} \mathbf{i})=\prod_{p=0}^{m}H(\sigma^{p} \mathbf{i}).$$  Since $\psi(R_{\mathbf {i}|{(m+1)}})\subset R_{(\sigma {\bf i})|m}$, it follows that
$$R_{\mathbf {i}|{(m+1)}}\subset f_{i_1, i_2} (R_{(\sigma {\bf i})|m})\subset  \bigcup_{j=1}^N f_{i_1, i_2}(B_j),$$
hence
$R_{\mathbf {i}|{(m+1)}}$ can be covered by
 $$G({\bf i})N\leq G({\bf i}) \cdot C_1\prod_{p=0}^{m-1}G(\sigma^{p+1} {\bf i})=C_1\prod_{p=0}^{m}G(\sigma^{p} {\bf i})$$
 balls of radius $\prod_{p=0}^{m}H(\sigma^{p} \mathbf{i})$.  Thus the proposition also holds for  $n=m+1$ and all ${\bf i}\in \Sigma_A$, as desired.
\end{proof}

\begin{pro}
\label{pro-4'}
 Let $k\in \{0, 1,\ldots, d-1\}$.  Let $G, H:\Sigma_A\to \R$ be defined as in \eqref{e-t9'}. Let $t$ be the unique real number so that $$P(\Sigma_A,\sigma,(\log G)+ t(\log H))=0.$$ Then
$
\overline{\dim}_B\Lambda \leq t.
$
\end{pro}

\begin{proof} Here we use similar arguments as that in the proof of Proposition \ref{pro-4}.
Write $g=\log G$ and $h=\log H$.
Define
$$
r_{\rm min}=\min_{{\bf i}\in \Sigma_A}h({\bf i}),\qquad r_{\rm max}=\max_{{\bf i}\in \Sigma_A}h({\bf i}).
$$
Then $0<r_{\rm \min}\leq  r_{\rm max}<1$. For $0<r<r_{\rm min}$, define
\begin{equation}
\label{e-t12}
\mathcal A_r=\left\{i_1\ldots i_n\in \Sigma^*_A:\; \sup_{x\in [i_1\ldots i_n]\cap \Sigma_A}S_nh(x)<\log r\leq \sup_{y\in [i_1\ldots i_{n-1}]\cap \Sigma_A}S_{n-1}h(y)\right\},
\end{equation}
where $\Sigma_A^*$ denotes the set of all finite words allowed in $\Sigma_A$.
Clearly $\{[I]:\; I\in \mathcal A_r\}$ is a partition of $\Sigma_A$. By Proposition \ref{pro-5.2}, there exists a constant $C_1>0$ such that for each $0<r<r_{\rm min}$,   every $I\in \mathcal A_r$ and $x\in [I]$, $R_I$ can be covered by
$$
C_1\exp(S_{|I|}g(x)) \leq C_1\exp\left(\sup_{y\in [I]}S_{|I|}g(y)\right)
$$
many balls of radius $$\exp(S_{|I|}h(x))\leq \exp\left(\sup_{y\in [I]}S_{|I|}h(y)\right)<r.$$
It follows that $\Lambda$ can be covered by
$$
C_1\sum_{I\in \mathcal A_r}\exp\left(\sup_{y\in [I]}S_{|I|}g(y)\right)
$$
many balls of radius $r$. Hence by Proposition \ref{pro-2.4},
$$
\overline{\dim}_B\Lambda\leq \limsup_{r\to 0}\frac{\sum_{I\in \mathcal A_r}\exp\left(\sup_{y\in [I]}S_{|I|}g(y)\right)}{\log (1/r)}= t.
$$
This completes the proof of the proposition.
\end{proof}

For $n\in \N$, applying Proposition \ref{pro-4'} to the mapping $\psi^n$  instead of $\psi$, we obtain the following.
\begin{pro}
\label{pro-upper'}
Let  $k\in \{0,\ldots, d-1\}$.
Then for each $n\in \N$,
$$
\overline{\dim}_B\Lambda\leq t_n,
$$
where $t_n$ is the unique number for which $P\left(\Sigma_A, \sigma^n, (\log G_n)+ t_n(\log H_n)\right)=0$, and $G_n, H_n$ are continuous functions on $\Sigma_A$ defined by
\begin{equation}
\label{e-6.9}G_n(y):=\frac{C\phi^k((D_{\Pi y}\psi^n)^{-1})}{\alpha_{k+1}((D_{\Pi y}\psi^n)^{-1}))^k},\qquad H_n(y):=\alpha_{k+1}((D_{\Pi y}\psi^n)^{-1})),
\end{equation}
with  $C=C_{\pmb M}$ being the constant in Lemma \ref{lem-6.5}.
\end{pro}

Now we are ready to prove Theorem \ref{thm-3}.

\begin{proof}[Proof of Theorem \ref{thm-3}] We follow the proof of Theorem \ref{thm-1} with slight modifications. Write $s=\dim_{S^*}(\Lambda)$. We may assume that $s<d$; otherwise we have nothing left to prove. Set $k=[s]$.
Let $\mathcal  G^s=\{g_n^s\}_{n=1}^\infty$ be the sub-additive potential on $\Lambda$ defined by
$$g_n^s(z)=\log \phi^s((D_{z} \psi^n)^{-1}).$$
Then $P(\Lambda, \psi, {\mathcal G}^s)=0$ by the definition of $\dim_{S^*}(\Lambda)$. Let $\widehat{\mathcal G}^s:=\{\widehat{g}_n^s\}_{n=1}^\infty$, where $\widehat{g}_n^s\in C(\Sigma_A)$ is  defined by
$$
\widehat{g}^s_n({\bf i})=g_n^s(\Pi{\bf i})=\log \phi^s((D_{\Pi{\bf i}} \psi^n)^{-1}),\quad {\bf i}\in \Sigma_A.
$$
Clearly, $\widehat{\mathcal G}^s$ is a sub-additive potential on $\Sigma_A$ and
$$
(\widehat{\mathcal G}^s)_*(m)=(\widehat{\mathcal G}^s)_*(m\circ \Pi^{-1}), \quad m\in {\mathcal M}(\Sigma_A, \sigma).
$$
  Since the factor map $\Pi:\Sigma_A\to \Lambda$ is onto and finite-to-one, by Lemma \ref{lem-2.3}, $m\to m\circ \Pi^{-1}$ is a surjective map from ${\mathcal M}(\Sigma_A, \sigma)$ to $\mathcal M(\Lambda, \psi)$ and moreover, $h_m(\sigma)=h_{m\circ \Pi^{-1}}(\psi)$ for $m\in \mathcal M (\Sigma_A, \sigma)$.  By the variational principle for the sub-additive pressure (see Theorem \ref{thm:sub-additive-VP}),
     \begin{eqnarray*}
 P\left(\Lambda, \psi, \mathcal{G}^{s}\right)    &=&  \sup \left\{h_{\mu}(\psi)+\left(\mathcal{G}^{s}\right)_{*}(\mu) :\; \mu \in \mathcal{M}(\Lambda, \psi)\right\}\\
     &=& \sup \left\{h_{m \circ \Pi^{-1}}(\psi)+\left(\mathcal{G}^{s}\right)_{*}\left(m \circ \Pi^{-1}\right) :\; m \in \mathcal{M}\left(\Sigma_{A}, \sigma\right)\right\} \\
     &=& \sup \left\{h_{m}(\sigma)+(\widehat{\mathcal{G}}^{s})_{*}(m) :\; m \in \mathcal{M}\left(\Sigma_{A}, \sigma\right)\right\} \\
     &=& P(\Sigma_{A},\sigma,\widehat{\mathcal{G}}^{s}).
  \end{eqnarray*}
It follows that  $P(\Sigma_{A},\sigma,\widehat{\mathcal{G}}^{s})=0$.

 By Proposition 2.2 in \cite{BanCaoHu10},
 $$
 \lim_{n\to \infty}\frac{1}{n}P\left(\Sigma_A, \sigma^n, \widehat{g}^s_n\right)=\inf_{n\geq 1}\frac{1}{n}P\left(\Sigma_A, \sigma^n, \widehat{g}^s_n\right)=P(\Sigma_A, \sigma, \widehat{\mathcal G}^s).
 $$
 Since $P(\Sigma_{A},\sigma,\widehat{\mathcal{G}}^{s})=0$,  the above equalities imply that for each $\epsilon>0$, there exists $N_\epsilon>0$ such that
 \begin{equation}
 \label{e-t8'}
 0\leq P\left(\Sigma_A, \sigma^n, \widehat{g}^s_n\right)\leq n\epsilon \quad\mbox{ for  }\; n>N_\epsilon.
 \end{equation}

For $n\in {\Bbb N}$, let $s_n$ be the unique real number so that $P(\Sigma_A, \sigma^n, v_n^{s_n})=0$, where $v_n^{s_n}$ is a continuous function on $\Sigma_A$ defined by
 \begin{eqnarray*}
 v_n^{s_n}(x)&=& \log G_n(x)+ s_n \log H_n(x)\\
 &=&\log \left( C\phi^k((D_{\Pi x}\psi^n)^{-1})\alpha_{k+1}((D_{\Pi x}\psi^n)^{-1})^{s_n-k}\right)
 \end{eqnarray*}
where $G_n, H_n$ are defined in \eqref{e-6.9} and $C=C_{\pmb M}$. By Proposition \ref{pro-upper'},
$$\overline{\dim}_B\Lambda\leq   s_n.$$
 If $s\geq s_n$ for some $n$, then $\overline{\dim}_B\Lambda\leq s_n\leq s$ and we are done. In what follows, we assume that $s<s_n$ for each $n$. Then for each $x\in \Sigma_A$,
  \begin{eqnarray*}
  \widehat{g}^s_n(x)-v_n^{s_n}(x)&=& -\log C +(s-s_n)\log \alpha_{k+1}((D_{\Pi x}\psi^n)^{-1})\\
  &\geq & -\log C +(s-s_n)\log \alpha_1((D_{\Pi x}\psi^n)^{-1})\\
  &\geq & -\log C+ n(s-s_n) \log \theta,
  \end{eqnarray*}
 where
$$
 \theta:=\max_{x\in \Sigma_A} \alpha_1((D_{\Pi x}\phi)^{-1})<1.
 $$
 Hence
 \begin{eqnarray*} P(\Sigma_A, \sigma^n,\widehat{g}^s_n)&=&P(\Sigma_A, \sigma^n,\widehat{g}^s_n)-P(\Sigma_A, \sigma^n, v_n^{s_n})\\
 &\geq& \inf_{x\in \Sigma_A} \left( \widehat{g}^s_n(x)-v_n^{s_n}(x)\right)\\
 &\geq&  -\log C+ n(s-s_n) \log \theta.
 \end{eqnarray*}
 where in the second inequality, we used \cite[Theorem 9.7(iv)]{Walters82}.
 Combining it with \eqref{e-t8'} yields that for $n\geq N_\epsilon$,
 $$
 n\epsilon\geq -\log  C + n(s-s_n) \log \theta,
 $$
 so
 $$
 s\geq s_n + \frac{\epsilon+ n^{-1} \log C}{\log \theta}\geq \overline{\dim}_B\Lambda+\frac{\epsilon+ n^{-1} \log C}{\log \theta}.
 $$
 Letting $n\to \infty$ and then $\epsilon\to 0$, we obtain $s\geq \overline{\dim}_B\Lambda$, as desired.
\end{proof}

In the remaining part of this section, we prove Theorem \ref{thm-4}. To this end, we need the following two lemmas, which are the analogues of Lemmas \ref{lem-4.1}-\ref{lem-est} for $C^1$ repellers.
\begin{lem}
\label{lem-5.1'}
Let $m$ be an ergodic $\sigma$-invariant Borel probability measure on $\Sigma_A$.  Set
\begin{equation}
\label{e-lambda'}
\lambda_i:=\lim_{n\to \infty} \frac{1}{n}\int \log\left( \alpha_i((D_{\Pi x}\psi^n)^{-1})\right)\; dm(x),\quad i=1,\ldots, d.
\end{equation}
Let $k\in \{0,\ldots, d-1\}$.  Write $u:=\exp(\lambda_{k+1})$. Then for $m$-a.e.~$x\in \Sigma_A$,
\begin{equation}
\label{e-gamma}
\limsup_{n\to \infty} \frac{1}{n}\log N_{u^n}(R_{x|n})\leq (\lambda_1+\cdots+\lambda_k)-k\lambda_{k+1},
\end{equation}
where $N_\delta(E)$ is the smallest integer $N$ for which $E$ can  be covered by $N$ closed balls of radius $\delta$.
\end{lem}

\begin{lem}
\label{lem-5.2}
Let $m$ be a Borel probability measure on $\Sigma_A$. Let $\rho, \epsilon \in (0,1)$. Then for $m$-a.e.~$x=(x_n)_{n=1}^\infty\in \Sigma_A$,
\begin{equation}
\label{e-Borel'}
m\circ\Pi^{-1} (B(\Pi x, 2\rho^n))\geq (1-\epsilon)^n\frac{m([x_1\ldots x_n])}{N_{\rho^n}(R_{x_1\ldots x_n})} \quad \mbox{ for  large enough }n.
\end{equation}
\end{lem}

The proofs of these two lemmas are essentially identical to that of Lemmas \ref{lem-4.1}-\ref{lem-est}. So we simply omit them.

\begin{proof}[Proof of Theorem \ref{thm-4}] Here we adapt the proof of Theorem \ref{thm-2}.
We may assume that $s:=\dim_{L^*}\mu<d$; otherwise there is nothing left to prove. Since $\Pi:\; \Sigma_A\to \Lambda$ is surjective and finite-to-one,  by Lemma \ref{lem-2.3}, there exists a $\sigma$-invariant ergodic measure $m$ on $\Sigma_A$ so that $m\circ \Pi^{-1}=\mu$ and $h_m(\sigma)=h_\mu(\psi)$.

Set $k=[s]$. Let $\lambda_i$, $i=1,\ldots, d$, be defined as in \eqref{e-lambda'}. Then by Definition \ref{de-6.2},
$$
h_m(\sigma)+\lambda_1+\cdots+\lambda_k+(s-k)\lambda_{k+1}=0.
$$
Let $u=\exp(\lambda_{k+1})$ and $\epsilon\in (0,1)$. Applying Lemma \ref{lem-5.2} (in which we take $\rho=u$) yields that for $m$-a.e.~$x=(x_n)_{n=1}^\infty\in \Sigma_A$,
\begin{equation}
\label{e-Borel'}
\mu (B(\Pi x, 2u^n))\geq (1-\epsilon)^n\frac{m([x_1\ldots x_n])}{N_{u^n}(R_{x_1\ldots x_n})} \quad \mbox{ for  large enough }n.
\end{equation}
It follows that for $m$-a.e.~$x=(x_n)_{n=1}^\infty\in \Sigma_A$,
\begin{eqnarray*}
\overline{d}(\mu, \Pi x)
&=&\limsup_{n\to \infty} \frac{\log\left(\mu (B(\Pi x, 2u^n))\right)}{n\log u} \\
&\leq &\frac{\log(1-\epsilon)}{\log u}+\limsup_{n\to \infty}\left(\frac{\log m([x_1\ldots x_n])}{n \log u} -\frac{\log N_{u^n}(\Pi([x_1\ldots x_n]))}{n\log u}\right)\\
&\leq& \frac{\log(1-\epsilon)-h_m(\sigma)-(\lambda_1+\cdots+\lambda_k)+k\lambda_{k+1}}{\log u}\\
&=& \frac{\log(1-\epsilon)+s\lambda_{k+1}}{\lambda_{k+1}},
\end{eqnarray*}
where in the third inequality, we used the Shannon-McMillan-Breiman theorem (cf.~\cite[Page.~93]{Walters82}) and Lemma \ref{lem-5.1'} (keeping in mind that $\log u=\lambda_{k+1}<0$). Letting $\epsilon\to 0$ yields the desired result. \end{proof}

\begin{rem}
{\rm
 There is  an alternative way to prove Theorem \ref{thm-3} in the case when  ${\pmb M}$ is an open set of $\R^d$.
In such case, we can construct a $C^1$ IFS $\{f_i\}_{i=1}^\ell$ on $\R^d$ so that $\Lambda$ is the projection of a shift invariant set  under the coding map associated with the IFS. Then we can use Theorem \ref{thm-1} to get an upper bound for $\overline{\dim}_B\Lambda$. An additional effort is then required to justify that this upper bound is indeed equal to $\dim_{S^*}\Lambda$.   The details of this approach will be given in a forthcoming survey paper.  
}
\end{rem}

\bigskip
\noindent {\bf Acknowledgement}. The research of Feng was partially supported by a HKRGC GRF grant and the Direct Grant for Research in CUHK. The research of Simon was partially supported by the
grant OTKA K104745.

\appendix

\section{}
In this part, we prove the following.
\begin{lem}
\label{lem-A1}
Let $(\A^\N,\sigma)$ be the one-sided full shift space over a finite alphabet $\A$.  Let $\nu\in \M(\A^\N, \sigma^m)$ for some $m\in \N$. Set $\mu=\frac{1}{m}\sum_{k=0}^{m-1}\nu\circ \sigma^{-k}$. Then $\mu\in \M(\A^{\N}, \sigma)$ and $h_\mu(\sigma)=\frac{1}{m}h_\nu(\sigma^m)$.
\end{lem}
\begin{proof} The lemma might be well-known. However we are not able to find a reference, so for the convenience of the reader we provide a self-contained proof.
The $\sigma$-invariance of $\mu$ follows directly from its definition, and we only need to prove that $h_\mu(\sigma)=\frac{1}{m}h_\nu(\sigma^m)$.

Clearly, $\nu\circ \sigma^{-k}\in \M(\A^\N, \sigma^m)$ for $k=0,\ldots, m-1$. We claim that $h_{\nu\circ \sigma^{-k}}(\sigma^m)=h_{\nu}(\sigma^m)$ for $k=1,\ldots, m-1$. Without loss of generality we prove this in the case when $k=1$.  For $n\in \N$, let ${\mathcal P}_n$ denote the partition of $\A^\N$ consisting of the $n$-th cylinders of $\A^\N$, i.e. ${\mathcal P}_n=\{ [I]:\; I\in \A^n\}$, and  set $\sigma^{-1} {\mathcal P}_n=\{\sigma^{-1}([I]):\; I\in \A^n\}$.  Then it is direct to see that
any element in  ${\mathcal P}_n$ intersects at most $\#\A$ elements in  $\sigma^{-1}{\mathcal P}_n$, and vice versa. Hence $$|H_\nu({\mathcal P}_n)-H_\nu(\sigma^{-1} {\mathcal P}_n|\leq \log \#\A; $$
see e.g. \cite[Lemma 4.6]{FengHu2009}. It follows that
$$
h_{\nu\circ \sigma^{-1}} (\sigma^m)=\lim_{n\to \infty} \frac{1}{n}H_{\nu\circ \sigma^{-1}} ({\mathcal P}_{nm})= \lim_{n\to \infty} \frac{1}{n}H_{\nu} ({\sigma^{-1}\mathcal P}_{nm})=\lim_{n\to \infty} \frac{1}{n}H_{\nu} ({\mathcal P}_{nm})=h_{\nu} (\sigma^m).
$$
This prove the claim.

By the affinity of the measure-theoretic entropy $h_{(\cdot)}(\sigma^m)$ (see \cite[Theorem 8.1]{Walters82}), we have
$$h_\mu(\sigma^m)=\frac{1}{m}\sum_{k=0}^{m-1} h_{\nu\circ \sigma^{-k}}(\sigma^m)=h_\nu(\sigma^m),$$
where the second equality follows from the above claim.   Hence we have  $h_\mu(\sigma)=\frac{1}{m}h_\mu(\sigma^m)=\frac{1}{m}h_\nu(\sigma^m)$.
\end{proof}

\section{Main notation and conventions}
\label{B}
For the reader's convenience, we summarize in Table \ref{table-1} the main notation and typographical conventions used in this paper.
\begin{table}[H]
\centering
\caption{Main notation and conventions}
\vspace{0.05 in}
\begin{footnotesize}
\begin{raggedright}
\begin{tabular}{p{1.5 in} p{4 in} }
\hline \rule{0pt}{3ex}
$\overline{\dim}_B$ & Upper box-counting dimension\\
$\dim_P$ & Packing dimension\\
$\{f_i\}_{i=1}^\ell$ & A $C^1$ IFS (Section \ref{S-1})\\
$(\Sigma, \sigma)$ & One-sided full shift over the alphabet $\{1,\ldots,\ell\}$\\
$\Pi:\; \Sigma\to K$ & Coding map associated with  $\{f_i\}_{i=1}^\ell$ (Section \ref{S-1})\\
$D_xf$ & Differential of $f$ at $x$\\
$\dim_{S}X$ & Singular dimension of $X$ w.r.t.~$\{f_i\}_{i=1}^\ell$ (cf.~Definition \ref{de-1})\\
$\dim_{L}m$ & Lyapunov dimension of $m$ w.r.t.~$\{f_i\}_{i=1}^\ell$ (cf.~Definition \ref{de-2})\\
$P(X,T, \{g_n\}_{n=1}^\infty)$ & Topological pressure of a sub-additive potential $\{g_n\}_{n=1}^\infty$ on a topological dynamical system $(X,T)$ (cf. Definition \ref{e-pressure})\\
$h_\mu(T)$ & Measure-theoretic entropy of $\mu$ w.r.t.~$T$\\
${\mathcal G}_*(\mu)$ & Lyapunov exponent of a sub-additive potential ${\mathcal G}$ w.r.t.~$\mu$ (cf. \eqref{e-N1})\\
$\alpha_i(T)$, $i=1,\ldots, d$ & The $i$-th singular value of $T\in {\Bbb R}^{d\times d}$ (Section \ref{S-2})\\
$\phi^s$ & Singular value function (cf.~\ref{e-singular})\\
$S_ng$ & $g+g\circ T+\cdots+g\circ T^{n-1}$ for $g\in C(X)$\\
$\mathcal G^s=\{g_n^s\}_{n=1}^\infty$ & (cf. \eqref{e-gn})\\
$\mathcal G^s_*(m)$ & $\lim_{n\to \infty} \frac{1}{n}\int g^s_n\; dm$\\
$\lambda_i(m)$, $i=1,\ldots,d$ & The $i$-th Lyapunov exponent of $m$ w.r.t.~$\{f_i\}_{i=1}^\ell$ (cf.~\eqref{e-lambda_1})\\
$N_\delta(E)$ & Smallest number of closed balls of radius $\delta$ required to cover $E$\\
$\dim_{S^*}\Lambda$ & Singularity dimension of $\Lambda$  w.r.t.~$\psi$ (cf. Definition \ref{de-6.1})\\
$\dim_{L^*}\mu$ & Lyapunov dimension of $\mu$  w.r.t.~$\psi$ (cf. Definition \ref{de-6.2})\\
\hline
\end{tabular}
\label{table-1}
\end{raggedright}
\end{footnotesize}
\end{table}

\end{document}